\documentclass[12pt, a4paper]{amsart}

\usepackage[hmargin=28.5mm, vmargin=25mm, includefoot, twoside]{geometry}
\usepackage[bookmarksopen=true]{hyperref}

\usepackage{amsfonts,amssymb,verbatim}
\usepackage{latexsym}
\usepackage{mathrsfs}
\usepackage{stmaryrd}
\usepackage{xspace}
\usepackage{enumerate, paralist}
\usepackage{graphicx}
\usepackage[all]{xy}
\usepackage{extarrows}
\usepackage{tabu}
\usepackage{tikz-cd}

\usepackage[noadjust]{cite}

\usepackage{txfonts, pxfonts}

\usepackage{amsthm}
\usepackage{amsmath}

\numberwithin{equation}{section}

\newtheorem{thm}{Theorem}[section]
 \newtheorem{cor}[thm]{Corollary}
 \newtheorem{lem}[thm]{Lemma}
 \newtheorem{prop}[thm]{Proposition}

\newtheorem{alphthm}{Theorem}			%letter numbering

 \theoremstyle{definition}
  \newtheorem{defn}[thm]{Definition}

  \newtheorem*{strategy*}{Strategy}

\newtheorem{alphdefn}[alphthm]{Definition}			%letter numbering

 \theoremstyle{remark}
 
  \newtheorem{ex}[thm]{Example}
   
\newtheorem*{claim*}{Claim}

\def\B{\mathfrak B}
\def\K{\mathfrak K}
\def\F{\mathfrak F}
\def\FP{\mathfrak{F}_{\rm proj}}
\def\H{\mathcal{H}}
\def\R{\mathcal{R}}

\def\E{\mathcal{E}}
\def\I{\mathcal{I}}
\def\L{\mathcal{L}}

\def\NN{\mathbb{N}}
\def\RR{\mathbb{R}}
\def\ZZ{\mathbb{Z}}
\def\CC{\mathbb{C}}

\def\supp{\mathrm{supp}}
\def\RANK{\mathrm{RANK}}
\def\rank{\mathrm{rank}}
\def\ker{\mathrm{Ker}}
\def\diam{\mathrm{diam}}
\def\ppg{\mathrm{prop}}

\begin{document}

\title{Fibring structures of ideals in Roe algebras and their $K$-theories}

\author{Zhijie Wang, Benyin Fu and Jiawen Zhang}

\address[Zhijie Wang]{School of Data Sciences, Jiaxing University, 899 Guangqiong Road, Jiaxing, 314000, China.}
\email{wangzhijie628@zjxu.edu.cn}

\address[Benyin Fu]{College of Statistics and Mathematics, Shanghai Lixin University of Accounting and Finance, Shangchuan Road 995, Shanghai, 201209, China.}
\email{fuby@lixin.edu.cn}

\address[Jiawen Zhang]{School of Mathematical Sciences, Fudan University, 220 Handan Road, Shanghai, 200433, China.}
\email{jiawenzhang@fudan.edu.cn}

\date{}
%\subjclass[2020]{47L20, 46L80, 51F30}
%\keywords{Roe algebras, Rank distribution, Geometric and ghostly ideals, Fibring structures, Property A}

\thanks{BF was partly supported by Natural Science Foundation of Shandong Province (No. ZR2024MA048). JZ was partly supported by NSFC (No. 12422107), and the National Key R{\&}D Program of China (No. 2022YFA100700).}

\baselineskip=16pt

\begin{abstract}
In this paper, we investigate the ideal structure of Roe algebras for metric spaces beyond the scope of Yu’s property A. Using the tool of rank distributions, we establish fibring structures for the lattice of ideals in Roe algebras and draw the border of each fibre by introducing the so-called ghostly ideals together with geometric ideals. We also provide coarse geometric criteria to ensure the coincidence of geometric and ghostly ideals and calculate their $K$-theories, which can be helpful to analyse obstructions to the coarse Baum-Connes conjecture on the level of ideals.
%To study the coarse Baum-Connes conjecture on the level of ideals, we provide coarse geometric criteria to ensure the coincidence of geometric and ghostly ideals and calculate their $K$-theories.
%, which also shows different phenomena from its uniform version.
%which also helps to simplify the ideal structure.
%
%We also calculate the $K$-theories of these ideals, helping to recover counterexamples to the coarse Baum-Connes conjecture. Furthermore, we provide coarse geometric criteria to ensure the coincidence of geometric and ghostly ideals, which shows different phenomenon from the case of uniform Roe algebras. 
%In this paper, we investigate the ideal structure of Roe algebras for general metric spaces beyond the scope of Yu’s property A. Using the tool of rank distributions, we establish fibring structures for ideals in Roe algebras and draw the border of each fibre by introducing the so-called ghostly ideals together with geometric ideals. Moreover, we provide coarse geometric criteria to ensure the coincidence of geometric and ghostly ideals and calculate their $K$-theories, which shows different phenomenon from the case of uniform Roe algebras.
\end{abstract}

%\date{\today}
\maketitle

%\parskip 4pt
%
%\noindent\textit{Mathematics Subject Classification} (2020): 47L20, 46L80, 51F30.\\
%\textit{Keywords: Roe algebras, Rank distribution, Geometric and ghostly ideals, Fibring structures, Property A}

\parskip 4pt

\begin{small}
\end{small}

\section{Introduction}\label{sec:intro}

Roe algebras are $C^*$-algebras defined for metric spaces, which encode their coarse geometric structures. These algebras were introduced by Roe \cite{Roe88} in his profound work on higher index theory, where he discovered that the $K$-theories of Roe algebras can serve as receptacles for higher indices of elliptic differential operators on open manifolds. 
Since then, there have been a number of excellent works around this topic (\emph{e.g.}, \cite{CWY13, HLS, WY12, Yu00}), which lead to significant progresses in topology, geometry and analysis (see, \emph{e.g.}, \cite{Roe96, Roe03}).

Let us recall definitions. Given a discrete metric space $(X,d)$ of bounded geometry and a separable infinite-dimensional Hilbert space $\H$, a bounded linear operator $T\in \B(\ell^2(X) \otimes \H)$ can be regarded as an $X$-by-$X$ matrix $[T(x,y)]_{x,y\in X}$ with $T(x,y) \in \B(\H)$. We say that $T$ has \emph{finite propagation} if there exists $R>0$ such that $T(x,y) = 0$ for any $x,y\in X$ with $d(x,y) > R$, and $T$ is \emph{locally compact} if each $T(x,y)$ is a compact operator. The \emph{Roe algebra} $C^*(X)$ of $X$ is defined to be the norm closure of the set of all finite propagation and locally compact operators on $\ell^2(X) \otimes \H$.

There is a uniform version of the Roe algebra once we replace $\H$ by the complex number $\CC$, which also plays a key role in higher index theory (see \cite{Spa09}). Comparing with uniform Roe algebras, the $K$-theories of Roe algebras are relatively easier to calculate, while their $C^*$-algebraic structures are more difficult to study.

Due to their importance, Chen and Wang investigated the ideal structures for (uniform) Roe algebras in \cite{CW01, CW, CW-05, CW-rank, Wan07}, and obtained a full description in the case of property A in \cite{CW-rank}. More precisely for a metric space $(X,d)$, they introduced a notion called rank distribution (see Definition \ref{def-rank distribution}), which is a family of non-negative integer valued functions on $X \times X$ satisfying a couple of axioms. Given a non-zero ideal $I$ in the Roe algebra $C^*(X)$, they associated to it a rank distribution $\R(I)$ (see Equality (\ref{EQ:R(I)})), which ``coarsely captures'' the rank of the matrix entries of operators in $I$. It turned out that rank distributions can be used to characterise geometric ideals (see Definition \ref{defn:geom ideal}), which are ideals where finite propagation operators are dense. Furthermore, they noticed that all ideals are geometric provided the space $X$ has property A and hence, established a complete picture for ideal structures of Roe algebras within the context of property A (see Proposition \ref{R=R(I(R))}).

Concerning the uniform case, the situation is relatively easier and it was shown in \cite{CW} that within the context of property A, all ideals are geometric and can be characterised using invariant open subsets of the Stone-\v{C}ech compactification $\beta X$ (see Section \ref{ssec:inv open}). Besides, it is worthwhile mentioning that recently Wang and the third-named author made the first step to the case beyond property A by introducing ghostly ideals in \emph{uniform} Roe algebras (see \cite{WZ23}).

%Concerning the uniform case, the situation is relatively easier since each matrix entry of an operator therein has rank either $0$ or $1$, which indicates that it suffices to consider $\{0,1\}$-valued functions in a rank distribution. It turns out that within the context of property A, all ideals are geometric and can be characterised using such simplified rank distributions (called ideals in the coarse structure, see \cite[Definition 4.1]{CW}). Besides, it is worthwhile mentioning that recently Wang and the third-named author made the first step to study the general case beyond property A. They introduced a new class of ideals in the uniform Roe algebras, called the ghostly ideals, and established a sandwiching structure for ideals using geometric and ghostly ones. 

While for Roe algebras, their structures are more complicated than the uniform version and hence their ideal structures are more difficult to study. In fact, over the last two decades, there seems no essential progress on their ideal structures beyond the scope of property A, and the general picture is far from clear.

In this paper, we investigate the ideal structure of Roe algebras beyond the scope of property A, trying to fill in the missing piece. Inspired by the work of rank distributions in \cite{CW-rank}, it is natural to consider the following map 
\begin{equation}\label{EQ:fibring of ideals}
\Psi: \{\text{non-zero ideals in the Roe algebra }C^*(X)\} \longrightarrow \{\text{rank distributions on }X\},
\end{equation}
given by $I \mapsto \R(I)$.
The correspondence of geometric ideals from \cite{CW-rank} shows that $\Psi$ is always surjective, while \emph{not} injective in general (see Example \ref{ex:ghost}). 
%In fact, the ideal of compact operators and that of ghost operators have the same associated rank distribution (see Example \ref{ex:ghost}). Here an operator is called a \emph{ghost operator} if it vanishes at infinity when regarding as a map on $X \times X$. They were introduced by Yu, and played a key role in the work \cite{HLS} to construct the first known counterexample to the coarse Baum-Connes conjecture, a central conjecture in higher index theory. 
Hence the map $\Psi$ provides a fibring structure of the lattice of non-zero ideals in the Roe algebra over the lattice of rank distributions, which suggests the following:

\begin{strategy*}
In order to study the ideal structure of the Roe algebra, it suffices to study the base space consisting of all rank distributions and each fibre $\Psi^{-1}(\R)$.
\end{strategy*}

One of the main contributions of this paper is providing structural results for both of the rank distributions and each fibre. 
Firstly, we introduce the following:

%To study the structure of each fibre $\Psi^{-1}(\R)$ for a given rank distribution $\R$, we introduce the notion of ghostly ideals in Roe algebras as follows:

\begin{alphdefn}[Definition \ref{defn:ghostly ideal}]
Given a rank distribution $\mathcal{R}$ on a metric space $X$, the associated \emph{ghostly ideal} in the Roe algebra $C^*(X)$ to $\R$ is defined to be 
\[
\tilde{I}(\mathcal{R}):=\{T\in C^*(X): \text{RANK}(T_\varepsilon)\in \mathcal{R} \text{ for any } \varepsilon>0\}.
\]
Here $T_\varepsilon$ is the $\varepsilon$-truncation of $T$ (see Definition \ref{defn:truncation}).
\end{alphdefn}

We show that $\tilde{I}(\R)$ is indeed an ideal in the Roe algebra $C^*(X)$ and belongs to the fibre $\Psi^{-1}(\R)$ (see Lemma \ref{lem:indeed an ideal} and \ref{lem:R for ghost}). Moreover, we obtain a sandwiching result (see Proposition \ref{sandwich lemma}), showing that the geometric ideal and the ghostly ideal provide the upper and lower bound for the fibre $\Psi^{-1}(\R)$, respectively.

As for the lattice of rank distributions, we introduce the associated invariant open subset $U_{\R}\subseteq \beta X$ to a given rank distribution $\R$ (see Equality (\ref{EQ:inv open subset})). Similar to the map $\Psi$ in (\ref{EQ:fibring of ideals}), we consider the following:
\begin{equation}\label{EQ:Phi}
\Phi: \{\text{rank distributions on }X\} \longrightarrow \{\text{non-empty invariant open subsets of } \beta X\}, 
\end{equation}
given by $\R \mapsto U_\R$. We show that $\Phi$ is always surjective but \emph{not} injective in general. Hence $\Phi$ provides a fibring structure for the base space in the strategy above, \emph{i.e.}, the lattice of rank distributions fibres over that of non-empty invariant open subsets of $\beta X$. Therefore, the study of the base space can be further reduced to each fibre $\Phi^{-1}(U)$. Given such a $U \subseteq \beta X$, we discover two special elements in the fibre $\Phi^{-1}(U)$, called the minimal and maximal rank distributions and denoted by $\R_{\min}(U)$ and $\R_{\max}(U)$, respectively (see Definition \ref{defn:min and max}). Moreover, we prove that they are the upper and lower bounds for the fibre $\Phi^{-1}(U)$ (see Proposition \ref{prop:sandwiching for rank dist}).
%, which provides a second sandwiching result.

In conclusion, the lattice of non-zero ideals in the Roe algebra $C^*(X)$ can be fibred over that of rank distributions, which can be further fibred over that of non-empty invariant open subsets of $\beta X$, and we have the following:

\begin{alphthm}[Proposition \ref{sandwich lemma} and Proposition \ref{prop:sandwiching for rank dist}]\label{thm:structure intro}
Let $X$ be a discrete metric space of bounded geometry. Concerning the ideal structure of the Roe algebra $C^*(X)$, we have:
\begin{enumerate}
 \item For a rank distribution $\R$ on $X$ and a non-zero ideal $I$ in $C^*(X)$ with $\R(I) = \R$, we have $I(\R) \subseteq I \subseteq \tilde{I}(\R)$. Here $I(\R)$ is the geometric ideal.
 \item For a non-empty invariant open subset $U\subseteq\beta X$ and a rank distribution $\R$ on $X$ with $U_\R = U$, we have $\R_{\min}(U) \subseteq \R \subseteq \R_{\max}(U)$. Moreover, $\R_{\min}(U) = \R_{\max}(U)$ \emph{if and only if} $U=X$.
\end{enumerate}
\end{alphthm}

Theorem \ref{thm:structure intro} makes the first step to study the ideal structure of the Roe algebra beyond the case of property A, following the strategy above. More precisely, once we can describe every ideal between the geometric and ghostly ideals for each rank distribution and every rank distribution between the minimal and maximal ones, then we will obtain a complete picture for the ideal structure of the Roe algebra. Comparing with the uniform case where only  $\Psi$ is needed since $\Phi$ is trivial, here more technical work concerning ranks is required.

On the other hand, our structural results shed new light on $K$-theories. The computation of $K$-theories for Roe algebras is a main focus in higher index theory since they contain higher indices of elliptic differential operators. One common approach is to consult the coarse Baum-Connes conjecture \cite{BC00, BCH94, HR95}, a central conjecture in higher index theory and have fruitful applications in topology, geometry and analysis \cite{Roe96}. It was shown in \cite{HLS} that this conjecture fails in general, where the minimal ghostly ideal (see Example \ref{ex:ghost}) plays a key role. We study the $K$-theories of general ghostly ideals in Roe algebras and prove the following:

\begin{alphthm}[Theorem \ref{thm:K-theory}]\label{thm:Ktheory intro}
If $X$ is a discrete metric space of bounded geometry which can be coarsely embedded into a Hilbert space, then for any rank distribution $\mathcal{R}$ on $X$, we have an isomorphism $(\iota_\R)_*:K_*(I(\mathcal{R}))\longrightarrow K_*(\tilde{I}(\mathcal{R}))$ for $\ast=0,1$ induced by the inclusion.
\end{alphthm}

Furthermore, it is important to learn when the geometric ideal $I(\R)$ coincides with the ghostly ideal $\tilde{I}(\R)$, which will simplify the ideal structure thanks to Theorem \ref{thm:structure intro}. However, this question is more difficult than only considering their $K$-theories as in Theorem \ref{thm:Ktheory intro}. Recall that the uniform case was already studied in \cite{WZ23}, where a notion called partial property A (see Definition \ref{defn:partial A}) was introduced as a sufficient condition to answer the question. However as we discover here, partial property A \emph{cannot} ensure $I(\R) = \tilde{I}(\R)$ for general rank distribution $\R$. More precisely, we prove the following (where partial property A trivially holds):

\begin{alphthm}[Theorem \ref{prop:counterexample}]\label{thm:counterexample intro}
For a sequence of expander graphs $\mathcal{G}_n=(V_n,E_n)$, denote $(X,d)$ their coarse disjoint union. For $\R:=\R_{\min}(\beta X)$, we have $I(\R) \neq \tilde{I}(\R)$ and moreover, $(\iota_{\R})_0: K_0(I(\R)) \to K_0(\tilde{I}(\R))$ is \emph{not} an isomorphism.
\end{alphthm}

Theorem \ref{thm:counterexample intro} indicates that ideal structures of Roe algebras are more complicated than their uniform counterparts and hence, more works are required here. It also shows the importance to study the fibring of rank distributions over non-empty invariant open subsets of $\beta X$ in Theorem \ref{thm:structure intro}(2).

Finally to remedy the defects revealed in Theorem \ref{thm:counterexample intro}, we provide the following:

\begin{alphthm}[see Theorem \ref{thm:partial A} for a precise statement]\label{thm:partial A intro}
Let $X$ be a discrete metric space of bounded geometry and $\R$ be a rank distribution on $X$. If $X$ has partial property A towards to $\beta X \setminus U_{\R}$ and finite rank projection valued functions controlled by $\R$ converges locally to the identity, then $I(\mathcal{R})=\tilde{I}(\mathcal{R})$.
\end{alphthm}

We notice that for a maximal rank distribution, the second condition holds automatically (see Lemma \ref{lem:condition for Rmax}) and hence in this case, partial property A \emph{does} imply that the geometric ideal is the same as the ghostly one.

\subsection*{Organisation}
In Section \ref{sec:pre}, we recall background knowledge in coarse geometry and groupoids. In Section \ref{sec:rank dist}, we recall the notions of rank distributions and geometric ideals from \cite{CW-rank} and provide an alternative picture by introducing a spatial version. Section \ref{sec:ghost} is devoted to the structural results (Theorem \ref{thm:structure intro}), divided into two parts. Finally in Section \ref{sec:relation}, we discuss their $K$-theories and determine when the geometric ideal coincides with the ghostly one by proving Theorems \ref{thm:Ktheory intro}, \ref{thm:counterexample intro} and \ref{thm:partial A intro}.

\subsection*{Acknowledgement} We would like to thank Jintao Deng and Qin Wang for helpful discussions and comments, and the anonymous referees for helpful suggestions.

\section{Preliminaries}\label{sec:pre}
Here we collect some basic notions including Roe algebras, groupoids, coarse groupoids, and crossed products.

\subsection{Roe algebras}

Let $X$ be a discrete metric space of bounded geometry. Here bounded geometry means that for any $r>0$, the number $\sup_{x\in X} \sharp B(x,r)$ is finite, where $\sharp Z$ denotes the cardinality of a set $Z$. Given $A,B \subseteq X \times X$, denote 
\[
A^{-1}:=\{(y,x): (x,y) \in A\}
\] 
and
\[
A\circ B:=\{(x,z): \exists y\in X \text{ such that } (x,y) \in A \text{ and } (y,z) \in B\}.
\]
A subset $E \subseteq X \times X$ is called an \emph{entourage} if there exists $r>0$ such that $d(x,y) \leq r$ for any $(x,y) \in E$. A subset $E\subseteq X\times X$ is called a \emph{partial translation} if it is an entourage and the projections $E \to X$ onto the first and second coordinates are injective.

Let $\H$ be a separable infinite dimensional Hilbert space, and $\B(\ell^2(X)\otimes \H)$ the set of all bounded operators on $\ell^2(X)\otimes \H$. Any operator $T\in \B(\ell^2(X)\otimes \H)$ has a matrix form
\[
T=[T(x,y)]_{x,y\in X} \quad \text{where} \quad T(x,y)\in \B(\H).
\]
Denote $\K(\H)$ the set of all compact operators on $\H$.
If $T(x,y)\in \K(\H)$ for any $(x,y)\in X\times X$, then $T$ is called \emph{locally compact}.  If the \emph{support of $T$},
\[
\supp(T):=\{(x,y)\in X\times X: T(x,y)\neq 0\},
\]
is an entourage, we say that $T$ has \emph{finite propogation} and its \emph{propagation} is
\[
\ppg(T):=\sup\{d(x,y): (x,y) \in \supp(T)\}.
\] 

\begin{defn}
The set of all locally compact and finite propagation operators on $\ell^2(X)\otimes \H$ forms a $\ast$-subalgebra in $\B(\ell^2(X)\otimes \H)$, denoted by $\mathbb{C}[X]$. Its norm closure in $\B(\ell^2(X)\otimes \H)$ is called the \emph{Roe algebra} of $X$, denoted by $C^*(X)$.
\end{defn}

Recall that a metric space $(X,d)$ can be \emph{coarsely embedded into a Hilbert space} if there exist a map $f: X \to \H$ and non-decreasing functions $\rho_{\pm}: [0,\infty) \to [0,\infty)$ with $\lim_{t\to +\infty}\rho_{\pm}(t) = +\infty$ such that $\rho_-(d(x,y)) \leq \|f(x) - f(y)\| \leq \rho_+(d(x,y))$ for any $x,y\in X$.

\subsection{Groupoids}

Recall that a \emph{groupoid} consists a set $G$, a subset $G^{(0)}$ called the set of \emph{units}, two maps $s,r:G\rightarrow G^{(0)}$ called the \emph{source} and the \emph{range} maps, respectively, a \emph{composition} law $(\gamma_1,\gamma_2)\in G^{(2)}\mapsto \gamma_1 \gamma_2\in G$, where
\[
G^{(2)}=\{(\gamma_1, \gamma_2) \in G \times G: s(\gamma_1)=r(\gamma_2)\},
\]
and an \emph{inverse} map $\gamma\mapsto \gamma^{-1}$. These operations satisfy certain rules, such as the facts that the composition is associative, the elements of $G^{(0)}$ act as units (see \cite[Definition 1.1]{Re} for a precise definition). For $U \subseteq G^{(0)}$, denote $G^U:=r^{-1}(U)$ and $G_U=s^{-1}(U)$. Also simplify the notation by $G^x=r^{-1}(x)$ and $G_x=s^{-1}(x)$ for $x\in G^{(0)}$.

A \emph{locally compact Hausdorff} groupoid is a groupoid $G$ equipped with a locally compact and Hausdorff topology such that the structure maps are continuous, where $G^{(0)}$ and $G^{(2)}$ have the induced topologies. Moreover, such a groupoid $G$ is \emph{\'{e}tale} if the range map (and hence the source map) is a local homeomorphism. In this case, each fibre $G^x$ (and $G_x$) is discrete with the induced topology.

\begin{ex}[pair groupoid]\label{ex:pair groupoid}
The \emph{pair groupoid} of a set $X$ is $X\times X$, whose unit space is $\{(x,x)\in X\times X: x\in X\}$, identified with $X$. The range and the source maps are the projections onto the first and the second coordinates, respectively. The composition rule is given by $(x,y)(y,z)=(x,z)$ for $x,y,z\in X$. When $X$ is discrete, then $X\times X$ is a locally compact Hausdorff \'{e}tale groupoid.
\end{ex}

Recall from \cite{ADR00} (see also \cite[Definition 5.6.13]{BO08}) that a locally compact, Hausdorff and \'{e}tale groupoid $G$ is \emph{(topologically) amenable} if for any $\varepsilon>0$ and compact $K \subseteq G$, there exists $f \in C_c(G)$ with range in $[0,+\infty)$ such that for any $\gamma \in K$ we have
\[
\left| \sum_{\alpha \in G_{r(\gamma)}} f(\alpha) - 1 \right| < \varepsilon  \quad \mbox{and} \quad \sum_{\alpha \in G_{r(\gamma)}} \left|f(\alpha) - f(\alpha\gamma)\right| < \varepsilon.
\]

\subsection{Coarse Groupoid} In \cite{STY}, Skandalis, Tu and Yu suggested a groupoid approach to study Roe algebras. For a discrete metric space $(X,d)$ of bounded geometry, they introduced a groupoid $G(X)$, called the coarse groupoid. As a topological space,
\[
G(X):=\bigcup\limits_{r\geq 0} \overline{E_r}^{\beta(X\times X)}\subseteq \beta(X\times X),
\]
with $E_r=\{(x,y)\in X\times X:d(x,y)\leq r\}$. Recall from Example \ref{ex:pair groupoid} that the pair groupoid $X\times X$ is endowed with the range and source maps $r(x,y)=x$ and $s(x,y)=y$. They can be extended to maps $G(X) \to \beta(X)$, still denoted by $r$ and $s$.

It was shown in \cite[lemma 2.7]{STY} that the pair of maps $(r,s):G(X)\rightarrow \beta(X)\times \beta(X)$ is injective. So $G(X)$ can be endowed with a groupoid structure, namely the \emph{coarse groupoid} of $X$, induced from the pair groupoid $\beta(X)\times \beta(X)$. It was also shown in \cite[Proposition 3.2]{STY} that $G(X)$ is a locally compact Hausdorff \'{e}tale groupoid with unit space $\beta(X)$. Moreover, from \cite[Theorem 5.3]{STY} we know that $G(X)$ is amenable if and only if $X$ has property A (see \cite[Definition 2.1]{Yu00} for the definition).

\subsection{Crossed Product}
Now we quickly review the definition of crossed products for groupoids acting on $C^*$-algebras. 
%For simplicity, here we only consider the case of locally compact, Hausdorff and \emph{\'{e}tale} groupoids. 
See \cite{AD, KS04, LeG} for more details.

Let $G$ be a locally compact, Hausdorff and \emph{\'{e}tale} groupoid with an action $\alpha$ on a $C^*$-algebra $A$.
%One can define the notion of continuous actions of $G$ on $C^*$-algebras and on Hilbert $C^*$-modules. For instance, a $G$-equivariant continuous field of $C^*$-algebra on $G^{(0)}$ is a $G$-algebra.
%Now suppose $G$ acts on $C^*$-algebra $A$. 
Let $C_c(G; A)$ be the space of functions with compact support $g\mapsto f(g)\in A_{r(g)}$ which are continuous in the sense of \cite{LeG}. Define the product and adjoint operations on $C_c(G; A)$ by
\[
f*g(\gamma)=\sum\limits_{\gamma'\in G^{r(\gamma)}}f(\gamma')\alpha_{\gamma'}(g(\gamma'^{-1}\gamma)),
\]
\[
f^*(\gamma)=\alpha_{\gamma}((f(\gamma^{-1}))^*).
\]
Define a norm on $C_c(G; A)$ by 
\[
\|f\|=\max\left\{\sup\limits_{x\in G^0}\sum\limits_{\gamma\in G^x}\|f(\gamma)\|, \, \sup\limits_{x\in G^0}\sum\limits_{\gamma\in G^x}\|f(\gamma^{-1})\|\right\}.
\]
The {\it full crossed product} $A\rtimes G$ is the enveloping $C^*$-algebra of the completion of $C_c(G; A)$ with respect to $\|\cdot\|$. 

Let $L^2(G; A)$ be the Hilbert $A$-module by taking the completion of the right $A$-module $C_c(G; A)$ with the inner product
\[
\langle \xi,\eta\rangle_x=\sum\limits_{\gamma\in G^x}\xi(\gamma)^*\eta(\gamma) \quad \text{for} \quad x \in G^{(0)},
\]
where the $A$-module is given by $(\xi a)(\gamma)=\xi(\gamma)a_{r(\gamma)}$ with $a=(a_x)_{x\in G^{(0)}}\in A$ as a continuous field of $C^*$-algebras. Then $C_c(G; A)$ acts on $L^2(G; A)$ by 
\[
(\Lambda(f)\xi)(\gamma)=\sum\limits_{\gamma'\in G^{r(\gamma)}}\alpha_{\gamma}(f(\gamma^{-1}\gamma'))\xi(\gamma').
\]
The {\it reduced crossed product} $A\rtimes_r G$ is the norm closure of $\Lambda(C_c(G; A))$ in the $C^*$-algebra $\mathcal{L}(L^2(G; A))$ of all bounded and adjointable operators on $L^2(G; A)$.

\begin{ex}\label{ex:coarse groupoid action}
For a discrete metric space $X$ of bounded geometry, take $A=\ell^\infty(X,\K(\H))$ where $\H$ is a separable infinite-dimensional Hilbert space. There is a natural action of the coarse groupoid $G(X)$ on $A$, and it was shown in \cite[Lemma 4.4]{STY} that $\ell^\infty(X,\K(\H))\rtimes_r G(X)$ is isomorphic to the Roe algebra $C^*(X)$.
\end{ex}

\section{Rank distributions and geometric ideals}\label{sec:rank dist}
In this section, we recall the notion of rank distributions from \cite{CW-rank}, which can be used to characterise geometric ideals of Roe algebras. For later use, we also introduce a notion called spatial rank distribution to provide an alternative description for rank distributions.

\subsection{The notion of rank distributions}\label{ssec:rank dist notion}

%Here we recall the notion of rank distributions introduced in \cite{CW-rank}.

Let us recall the following:

\begin{defn}[{\cite[Definition 4.1]{CW-rank}}]
For a metric space $(X,d)$, a non-negative integer valued function $\alpha:X\times X\rightarrow \mathbb{Z}_{\geq 0}$ is called a {\it rank function} if the support of $\alpha$, denoted by $\supp(\alpha):=\{(x,y): \alpha(x,y) \neq 0\}$, is an entourage.
\end{defn}

The set of all rank functions on $X\times X$ can be naturally ordered as follows:
\[
\alpha_1\leq \alpha_2\quad \text{if and only if} \quad  \alpha_1(x,y)\leq \alpha_2(x,y) \quad \text{for any} \quad (x,y)\in X\times X.
\] 
If $\alpha$ is a rank function and $E\subseteq X\times X$ is an entourage, define $E\cdot \alpha$ and $\alpha\cdot E$ by 
\[
(E\cdot\alpha)(x,y)=\sum\limits_{z\in X}\chi_E(x,z)\alpha(z,y),
\]
\[
(\alpha\cdot E)(x,y)=\sum\limits_{z\in X}\alpha(x,z)\chi_E(z,y),
\]
where $\chi_E$ is the characteristic function of $E$. Clearly $E\cdot \alpha$ and $\alpha\cdot E$ are rank functions.

\begin{defn}[{\cite[Definition 4.2]{CW-rank}}]\label{def-rank distribution}
Let $(X,d)$ be a discrete metric space of bounded geometry. A {\it rank distribution} $\R$ on $X$ is a collection of rank functions on $X\times X$ satisfying the following conditions:
\begin{enumerate}
  \item For $\alpha, \beta\in \mathcal{R}$ and an entourage $E\subseteq X\times X$, then $\alpha+\beta$, $\alpha^t$, $E\cdot \alpha$ and $\alpha\cdot E$ belong to $\mathcal{R}$, where $(\alpha+\beta)(x,y)=\alpha(x,y)+\beta(x,y)$ and $\alpha^t(x,y)=\alpha(y,x)$.
  \item If $\alpha$ is a rank function with finite support, then $\alpha\in \mathcal{R}$.
  \item If $\alpha$ is a rank function and $\alpha\leq \beta$ for some $\beta\in \R$, then $\alpha\in \R$.
\end{enumerate}
%For later use, we denote $n \alpha$ for the point-wise addition of $n$'s $\alpha$, where $n\in \NN$.
\end{defn}

To relate a rank distribution to an ideal in the Roe algebra, let us follow the approach in \cite{CW-rank} and start by recalling the notion of $\varepsilon$-truncation introduced therein. 

Let $(\H, \langle \cdot, \cdot \rangle)$ be a separable infinite-dimensional Hilbert space. For non-zero vectors $\xi, \eta \in \H$, let $\xi\otimes \eta$ be the rank-one operator on $\H$ defined by
\[
(\xi\otimes\eta)x=\langle x,\eta\rangle\xi \quad \text{for} \quad  x\in \H.
\]
Given a compact operator $T \in \K(\H)$, it has a canonical decomposition
\[
T=\sum\limits_{i} \lambda_i(T) u_i\otimes v_i,
\]
where $\{\lambda_i(T)\}_{i \in \NN} \subseteq \RR$ is non-increasing and tends to zero, called the \emph{singular values} of $T$, and $\{u_i\}, \{v_i\}$ are orthonormal sets in $\H$. We record the following known facts (see, \emph{e.g.}, \cite[Chapter 11, P. 244]{BS87}):
% concerning singular values for later use: 

%For singular values, we recall the following basic properties (see [] for details) which we will use in next section.

\begin{lem}\label{property-singular value}
Given $T,S\in \K(\H)$ and $R \in \B(\H)$, we have:
\begin{enumerate}
\item [(1)] $\lambda_i(T)=\lambda_i(T^*)$, where $T^*$ is the adjoint operator of $T$;
\item [(2)] $\lambda_i(RT)\leq \|R\|\lambda_i(T)$ and $\lambda_i(TR)\leq \|R\| \cdot \lambda_i(T)$ for each $i\in \NN$;
\item [(3)] $\lambda_{i+j+1}(T+S)\leq \lambda_{i+1}(T)+\lambda_{j+1}(S)$ for each $i,j\in \NN$.
\end{enumerate}
\end{lem}

Given $T\in \K(\H)$ and $\varepsilon>0$, define the \emph{$\varepsilon$-truncation} of $T$ to be
\[
T_\varepsilon=\sum\limits_{i \in \NN:\lambda_i(T)\geq \varepsilon}\lambda_i(T) u_i\otimes v_i.
\]
It is obvious that $T_\varepsilon$ is a finite rank operator on $\H$.

\begin{defn}[{\cite[Definition 3.1]{CW-rank}}]\label{defn:truncation}
The $\varepsilon$-{\it truncation} of $T\in C^*(X)$ is defined to be the operator $T_\varepsilon=[T_\varepsilon(x,y)]_{x,y\in X}$, where $T_\varepsilon(x,y):=T(x,y)_\varepsilon$ is defined above.
\end{defn}

To associate a rank distribution to a non-zero ideal in the Roe algebra, let $T=[T(x,y)]_{x,y\in X}$ be an operator in $C^*(X)$ such that $T(x,y)$ has finite rank for all $(x,y)\in X\times X$. Define a function $\text{RANK}(T):X\times X\rightarrow \mathbb{Z}_{\geq 0}$ by 
\[
\text{RANK}(T)(x,y)=\text{rank}(T(x,y)) \quad \text{for} \quad (x,y) \in X \times X,
\]
where $\text{rank}(T(x,y))$ denotes the rank of $T(x,y)$.
Given a non-zero ideal $I$ in $C^*(X)$, define
\begin{equation}\label{EQ:R(I)}
\mathcal{R}(I):=\{\text{RANK}(T_\varepsilon):T\in I, \varepsilon>0\}.
\end{equation}
It was shown in \cite[Proposition 4.5]{CW-rank} that $\mathcal{R}(I)$ is a rank distribution on $X$.

Conversely, Chen and Wang introduced the following ideal associated to a given rank distribution. Denote $\F(\H)$  the set of all finite-rank operators on $\H$.

\begin{defn}[\cite{CW-rank}]\label{defn:I(R)}
Given a rank distribution $\mathcal{R}$ on $X$, denote 
\[
C_{\mathcal{R}}:=\{T\in \mathbb{C}[X]:T(x,y)\in \F(\H) \text{ for all } x,y\in X \text{ and } \RANK(T)\in \mathcal{R}\}.
\]
The closure of $C_\mathcal{R}$ in $C^*(X)$ is a non-zero ideal in the Roe algebra $C^*(X)$, denoted by $I(\mathcal{R})$, called the \emph{ideal associated to $\R$}.
\end{defn}

It is clear that the ideal $I(\mathcal{R})$ is geometric in the following sense:

\begin{defn}\label{defn:geom ideal}
An ideal $I$ in the Roe algebra $C^*(X)$ is called \emph{geometric} if the finite propagation operators in $I$ are dense in $I$, \emph{i.e.}, $I=\overline{I\cap \mathbb{C}[X]}$.
\end{defn}

The following is the main result in \cite{CW-rank}:

\begin{prop}[{\cite[Theorem 4.6]{CW-rank}}]\label{R=R(I(R))}
Given a rank distribution $\mathcal{R}$ on $X$, then $\mathcal{R}=\mathcal{R}(I(\mathcal{R}))$. Conversely, given a non-zero  geometric ideal $I$ in $C^*(X)$, then $I=I(\mathcal{R}(I))$.
\end{prop}

Proposition \ref{R=R(I(R))} shows that geometric ideals are determined by rank distributions, while this is not the case for general ideals:

\begin{ex}\label{ex:ghost}
Consider the ideal of compact operators in $\B(\ell^2(X) \otimes \H)$. It is easy to calculate that the associated rank distribution consists of all finitely supported rank functions, denoted by $\R_{\mathrm{fin}}$. On the other hand, consider the ideal $\I_{\mathrm{G}}$ consisting of all ghost operators in $C^*(X)$. Here recall that an operator $T \in C^*(X)$ is called a \emph{ghost} (introduced by G. Yu) if for any $\varepsilon>0$, there exists a finite $F \subseteq X$ such that for any $(x,y) \notin F \times F$, then $\|T(x,y)\| < \varepsilon$. Direct calculations show that the rank distribution associated to $\I_{\mathrm{G}}$ is also $\R_{\mathrm{fin}}$. However, we know from \cite{RW14} that in general there exist non-compact ghost operators (\emph{e.g.}, for expander graphs).
\end{ex}

\subsection{Spatial rank distributions}\label{ssec:spatial rank}
Now we provide a new characterisation of rank distributions for later use, inspired by \cite{CW}. Recall that for a metric space $X$, Chen and Wang in \cite{CW} introduced the notion of ideals in $X$ and show that this is indeed another description for geometric ideals in the uniform Roe algebra. Here we show a parallel result in the case of Roe algebras.

\begin{defn}\label{spatial-rank-distribution}
A \emph{spatial rank distribution} $\mathcal{L}$ on a discrete metric space $X$ of bounded geometry is a collection of functions $\varphi: X\rightarrow \mathbb{Z}_{\geq 0}$ satisfying the following:
\begin{enumerate}
\item [(1)] If $\varphi_1, \varphi_2\in \mathcal{L}$, then $\varphi_1+\varphi_2\in \mathcal{L}$.
\item [(2)] If $\supp(\varphi):=\{x\in X: \varphi(x)\neq 0\}$ is finite, then $\varphi\in \mathcal{L}$.
\item [(3)] If $\varphi_1\in \mathcal{L}$ and $\varphi_2\leq \varphi_1$ point-wise, then $\varphi_2\in \mathcal{L}$.
\item [(4)] For $\varphi\in \mathcal{L}$ and $r>0$, the following function $\widetilde{\varphi}:X\rightarrow \mathbb{Z}_{\geq 0}$ is contained in $\mathcal{L}$:
\begin{equation}\label{EQ:tilde phi}
\widetilde\varphi(x)=\max_{y:d(x,y)\leq r}\varphi(y).
\end{equation}
\end{enumerate}
\end{defn}

To relate a spatial rank distribution to a given rank distribution $\mathcal{R}$ on $X$, we define a function 
\[
\varphi_\alpha: X \to \ZZ_{\geq 0} \quad \text{by} \quad \varphi_\alpha(x):=\max_{y\in X} \alpha(x,y)
\]
for $\alpha \in \R$. Then we define
\begin{equation}\label{EQ:L(R)}
\mathcal{L}(\mathcal{R}):=\{\varphi_\alpha: \alpha \in \R\}.
\end{equation}

To see that $\L(\R)$ is a spatial rank distribution, we need an auxiliary lemma. For $\varphi: X \to \mathbb{Z}_{\geq 0}$, define $\Delta_\varphi: X\times X\rightarrow \mathbb{Z}_{\geq 0}$ by $\Delta_\varphi(x,y)=\varphi(x)$ if $x=y$ and $\Delta_\varphi(x,y)=0$ if $x\neq y$. We notice the following:

\begin{lem}\label{lem:spatial rank lem}
Given a rank distribution $\mathcal{R}$ on $X$ and $\varphi \in \L(\R)$, we have $\Delta_\varphi \in \R$.
\end{lem}

\begin{proof}
Assume that $\varphi = \varphi_\alpha$ for some $\alpha \in \R$. Denote $E:=\supp(\alpha)$. Then for any $x\in X$, we have
\[
(\alpha \cdot E^{-1})(x,x) = \sum_{y\in X} \alpha(x,y) \chi_E(x,y) \geq \max_{y\in X} \alpha(x,y) = \Delta_\varphi(x,x).
\]
Hence we have $\Delta_\varphi \leq \alpha \cdot E^{-1}$, which belongs to $\R$.
%By \cite[Lemma 2.8]{STY}, we can decompose $E=E_1\cup E_2\cup \cdots\cup E_n$ into partial translations with $E_i\cap E_j=\emptyset$ for $i\neq j$. Set $\alpha_i:=(\alpha \cdot \chi_{E_i}) \cdot E_i^{-1} \in \R$ for $i=1,2,\cdots, n$, whose support is contained in the diagonal of $X \times X$. Here $\alpha \cdot \chi_{E_i}$ denotes the point-wise multiplication. Then $\sum_{i=1}^n \alpha_i \in \R$ with $\Delta_\varphi \leq \sum_{i=1}^n \alpha_i$, and hence $\Delta_\varphi \in \R$.
\end{proof}

Now we have the following:
\begin{lem}
Given a rank distribution $\mathcal{R}$ on $X$, then $\mathcal{L}(\mathcal{R})$ is a spatial rank distribution.
\end{lem}

\begin{proof}
Condition (1) in Definition \ref{spatial-rank-distribution} follows from Lemma \ref{lem:spatial rank lem}, and condition (2) and (3) hold trivially. Concerning condition (4): Given $\varphi = \varphi_\alpha \in \L(\R)$ for some $\alpha \in \R$, take $r:=\sup\{d(x,y): (x,y) \in \supp(\alpha)\}$. Then we have $\alpha \cdot E_r\in \mathcal{R}$ where $E_r=\{(x,y):d(x,y)\leq r\}$. It is easy to see that $\Delta_{\widetilde{\varphi}} \leq \alpha \cdot E_r$ for $\widetilde{\varphi}$ defined in (\ref{EQ:tilde phi}), and hence we conclude condition (4) in Definition \ref{spatial-rank-distribution}.
%Given $\varphi = \varphi_\alpha \in \L(\R)$ for some $\alpha \in \R$, note that $\varphi(x)=\Delta_\varphi(x,x)$ where $x\in X$ and $\Delta_\varphi \in \R$ thanks to Lemma \ref{lem:spatial rank lem}. Hence conditions (1)-(3) in Definition \ref{spatial-rank-distribution} hold. Taking $r:=\sup\{d(x,y): (x,y) \in \supp(\alpha)\}$, we have $\alpha \cdot E_r\in \mathcal{R}$ where $E_r=\{(x,y):d(x,y)\leq r\}$. It is easy to see that $\Delta_{\widetilde{\varphi}} \leq \alpha \cdot E_r$ for $\widetilde{\varphi}$ defined in (\ref{EQ:tilde phi}), and hence we conclude condition (4) in Definition \ref{spatial-rank-distribution}.
\end{proof}

Conversely, given a spatial rank distribution $\mathcal{L}$ on $X$, we define $\mathcal{R}(\mathcal{L})$ to be the collection of rank functions $\alpha: X\times X\rightarrow \ZZ_{\geq 0}$ such that there exists $\varphi\in \mathcal{L}$ with $\alpha(x,y)\leq \varphi(x)$ for all $x,y\in X$. It is easy to check that $\mathcal{R}(\mathcal{L})$ is a rank distribution. Moreover, we have the following, which shows that there is a one-to-one correspondence between rank distributions and their spatial version:

\begin{prop} \label{relations between rank distributions}
Given a rank distribution $\R$ on $X$, we have $\mathcal{R}=\mathcal{R(\mathcal{L}(\mathcal{R}))}$. Given a spatial rank distribution $\L$ on $X$, we have $\mathcal{L}=\mathcal{L(\mathcal{R}(\mathcal{L}))}$.
\end{prop}

\begin{proof}
We only prove the first relation, and the second is similar.

Given $\alpha\in \mathcal{R}$, we have $\alpha(x,y) \leq \varphi_\alpha(x)$ for any $x\in X$. Hence $\alpha \in \mathcal{R(\mathcal{L}(\mathcal{R}))}$. Conversely, given $\alpha \in \mathcal{R(\mathcal{L}(\mathcal{R}))}$, there exists $\tilde{\alpha} \in \R$ such that $\alpha(x,y)\leq \varphi_{\tilde{\alpha}}(x)$ for any $x,y\in X$. By Lemma \ref{lem:spatial rank lem}, we have $\Delta_{\varphi_{\tilde{\alpha}}} \in \R$. Setting $E:=\text{Supp}(\alpha)$, we have $\alpha\leq \Delta_{\varphi_{\tilde{\alpha}}}\cdot E$ and hence, we obtain $\alpha\in \mathcal{R}$.
\end{proof}

Combining with Proposition \ref{R=R(I(R))}, we obtain that non-zero geometric ideals in the Roe algebra of $X$ can also be characterised by spatial rank distributions on $X$.

\section{Ghostly ideals and fibring structures}\label{sec:ghost}

This section is devoted to studying the fibring structures of the lattice of non-zero ideals in Roe algebras and aiming to prove Theorem \ref{thm:structure intro}. Recall from (\ref{EQ:fibring of ideals}) that we have the following map:
%Section \ref{ssec:rank dist notion} that we can use rank distributions to ``parametrise'' ideals, which gives rise to the fibring map $\Psi$ defined in (\ref{EQ:fibring of ideals}):
\begin{equation*}
\Psi: \{\text{non-zero ideals in }C^*(X)\} \longrightarrow \{\text{rank distributions on }X\}, \quad I \mapsto \R(I).
\end{equation*}
Proposition \ref{R=R(I(R))} shows that $\Psi$ is surjective, while Example \ref{ex:ghost} shows that it is \emph{not} injective in general. 
Following the strategy in Section \ref{sec:intro}, we divide the section into two parts to study the base space of rank distributions and each fibre separately.

%Hence $\Psi$ can be regarded as a fibration of all ideals in $C^*(X)$ over all rank distributions on $X$. 
%Hence to study the ideal structure of the Roe algebra $C^*(X)$, it suffices to study the structures of all rank distributions on $X$ and each fibre $\Psi^{-1}(\R)$. We will treat them separately in the following.

%, where the former will be explored in Section \ref{ssec:inv open} and the latter in Section \ref{sec:ghost}.

\subsection{Ghostly ideals and structure of fibres}\label{ssec:ghostly ideals}

%Recall from the end of Section \ref{ssec:rank dist notion} that the ideal structure of the Roe algebra can be decomposed into the structure of all rank distributions on $X$ and the study of each fibre $\Psi^{-1}(\R)$, where $\Psi$ is defined in (\ref{EQ:fibring of ideals}) and $\R$ is a rank distribution. We already study rank distributions in Section \ref{ssec:inv open} and obtain a sandwiching result, \emph{i.e.}, Proposition \ref{prop:sandwiching for rank dist}. 

Here we focus on the fibre $\Psi^{-1}(\R)$ for a given rank distribution $\R$, and aim to provide a border for $\Psi^{-1}(\R)$. To achieve, we introduce a new class of ideals in Roe algebras, called the ghostly ideals.
%, which are inspired by the work \cite{WZ23}.

%Let $X$ be a metric space and $\mathcal{R}$ be a rank distribution on $X$. In this section, we would like to define the ghost ideal $\widetilde{I}(\mathcal{R})$ with respect to $\mathcal{R}$ for Roe algebra $C^*(X)$. Moreover, if $X$ is coarsely embeddable into a Hilbert space, then for any rank distribution $\mathcal{R}$ on $X$, the inclusion map from $I(R)$ to $\widetilde{I}(\mathcal{R})$ induces an isomorphism on $K$-theory level.

\begin{defn}\label{defn:ghostly ideal}
Given a rank distribution $\mathcal{R}$ on $X$, we define the \emph{associated ghostly ideal} to be 
$\tilde{I}(\mathcal{R}):=\{T\in C^*(X): \text{RANK}(T_\varepsilon)\in \mathcal{R} \text{ for any } \varepsilon>0\}$.
\end{defn}

%Let us verify that $\tilde{I}(\mathcal{R})$ is indeed an ideal:

\begin{lem}\label{lem:indeed an ideal}
Given a rank distribution $\mathcal{R}$ on $X$, $\tilde{I}(\mathcal{R})$ is an ideal in $C^*(X)$.
\end{lem}

\begin{proof}
Given $T,S\in \tilde{I}(\mathcal{R})$ and $\varepsilon>0$, we consider $T+S$ and $T^*$. For $x,y\in X$, set 
\[
i_0:=\max\left\{i \in \NN: \lambda_i(T(x,y))\geq \dfrac{\varepsilon}{2}\right\} \quad \text{and} \quad j_0:=\max\left\{j:\lambda_j(S(x,y)) \geq \dfrac{\varepsilon}{2}\right\}.
\]
Using Lemma \ref{property-singular value}(3), we have 
\[
\lambda_{i_0+j_0+1}((T+S)(x,y)) \leq \lambda_{i_0+1}(T(x,y))+\lambda_{j_0+1}(S(x,y)) <\varepsilon.
\]
This implies that 
\[
\text{rank}\left((T+S)_\varepsilon(x,y)\right) \leq i_0+j_0=\text{rank}\left(T_{\frac{\varepsilon}{2}}(x,y)\right)+\text{rank}\left(S_{\frac{\varepsilon}{2}}(x,y)\right)
\]
for any $x,y\in X$. Hence 
\[
\text{RANK}\left((T+S)_\varepsilon\right) \leq \text{RANK}(T_{\frac{\varepsilon}{2}})+\text{RANK}(S_{\frac{\varepsilon}{2}}),
\]
which shows that $\text{RANK}\left((T+S)_\varepsilon\right)\in \mathcal{R}$. Therefore, $T+S\in \tilde{I}(\mathcal{R})$.

As for $T^*$, note that for any $x,y\in X$ we have $\lambda_i(T(x,y))=\lambda_i(T(x,y)^*)$ for $i\in \NN$ from Lemma \ref{property-singular value}(1). Hence we have
\[
\text{RANK}\left((T^*)_\varepsilon\right)=\left(\text{RANK}(T_\varepsilon)\right)^t \in \mathcal{R},
\]
which shows $T^*\in \tilde{I}(\mathcal{R})$.

To see that $\tilde{I}(\mathcal{R})$ is closed, take a sequence $T_n\in \tilde{I}(\mathcal{R})$ converging to $T' \in C^*(X)$ as $n \to \infty$. Given $\varepsilon'>0$, there exists $N \in \NN$ such that $\|T_N-T'\|\leq \varepsilon'/2$. For any $(x,y)\in X\times X$, set $i'_0:=\max\{i: \lambda_i(T_N(x,y))\geq \varepsilon'/2\}$. Then by Lemma \ref{property-singular value}(3), we have 
\[
\lambda_{i'_0+1}(T'(x,y))\leq \lambda_1(T'(x,y)-T_N(x,y))+\lambda_{i'_0+1}(T_N(x,y))<\varepsilon'.
\]
Hence 
\[
\text{rank}(T'(x,y)_{\varepsilon'}) \leq i'_0=\text{rank}\left(T_N(x,y)_{\varepsilon'/2}\right),
\] 
showing that $\text{RANK}(T'_{\varepsilon'}) \leq \text{RANK}((T_N)_{\varepsilon'/2})$. Therefore, $T'\in \tilde{I}(\mathcal{R})$. 

Finally, given $\hat{T} \in \tilde{I}(\mathcal{R})$ and $\hat{S}\in C^*(X)$, we need to show that $\hat{T}\hat{S}$ and $\hat{S}\hat{T}$ belong to $\tilde{I}(\mathcal{R})$. By the analysis above, we only need consider $\hat{S}\hat{T}$ for 
%Since $\tilde{I}(\mathcal{R})$ is closed under taking adjoints and norm closed, we only consider $\hat{S}\hat{T}$ for $\hat{S}$ having finite propagation. Moreover since $\tilde{I}(\mathcal{R})$ is close under taking addition, we can further assume that 
$\hat{E}:=\text{supp}(\hat{S})$ being a partial translation. Using Lemma \ref{property-singular value}(2), it is easy to check that
%For any $x,y\in X$, we have
%\[
%\lambda_i(\hat{S}(x,y)\hat{T}(x,y)) \leq \|\hat{S}(x,y)\| \cdot \lambda_i(\hat{T}(x,y))\leq \|S\| \cdot \lambda_i(\hat{T}(x,y))
%\]
%by Lemma \ref{property-singular value}(2).
\[
\text{RANK}((\hat{S}\hat{T})_{\hat{\varepsilon}}) \leq \hat{E} \cdot \text{RANK}(\hat{T}_{\hat{\varepsilon}/\|\hat{S}\|}) \quad \text{for any} \quad  \hat{\varepsilon}>0.
\] 
Hence we obtain $\text{RANK}((\hat{S}\hat{T})_{\hat{\varepsilon}}) \in \mathcal{R}$, which implies that $\hat{S}\hat{T}\in \tilde{I}(\mathcal{R})$. 
\end{proof}

%Recall that given an ideal $I$ in $C^*(X)$, we associate a rank distribution defined in (\ref{EQ:R(I)}) by $\mathcal{R}(I)=\{\text{RANK}(T_\varepsilon):T\in I \text{ and } \varepsilon>0\}$, and the fibration map $\Phi$ is given by $I \mapsto \R(I)$.
 
The following shows that the ghostly ideal $\tilde{I}(\R)$ belongs to the fibre $\Psi^{-1}(\R)$:

\begin{lem}\label{lem:R for ghost}
Given a rank distribution $\R$ on $X$, we have $\mathcal{R}(\tilde{I}(\mathcal{R}))=\mathcal{R}$.
\end{lem}

\begin{proof}
By definition, for any $T\in \tilde{I}(\mathcal{R})$ and $\varepsilon>0$, then $\RANK(T_\varepsilon)\in \mathcal{R}$. Hence $\mathcal{R}(\tilde{I}(\mathcal{R}))\subseteq \mathcal{R}$. Conversely, it is clear that $I(\R) \subseteq \tilde{I}(\R)$ and hence, $\mathcal{R}(\tilde{I}(\mathcal{R})) \supseteq \mathcal{R}(I(\mathcal{R})) = \R$, where the last equality comes from Proposition \ref{R=R(I(R))}.
\end{proof}

Now we are in the position to provide the following sandwiching result:

%Now we have the following sandwiching result, showing that the geometric and the ghostly ideal provide the lower and upper bound, respectively, for ideals in the Roe algebra with the same associated rank distribution.

\begin{prop}\label{sandwich lemma}
For a rank distribution $\R$ on $X$ and a non-zero ideal $I$ in $C^*(X)$ with $\R(I) = \R$, we have $I(\R) \subseteq I \subseteq \tilde{I}(\R)$. Hence for the fibre $\Psi^{-1}(\R)$, the geometric ideal $I(\R)$ is its lower bound and the ghostly ideal $\tilde{I}(\R)$ is its upper bound.
%For any ideal $I$ in $C^*(X)$, we have $I(\mathcal{R}(I))\subseteq I\subseteq \tilde{I}(\mathcal{R}(I))$. Therefore, given a rank distribution $\R$ on $X$ and any ideal $I$ in $C^*(X)$ with $\R(I) = \R$, we have $I(\R) \subseteq I \subseteq \tilde{I}(\R)$. In other words, $I(\R)$ is the lower bound of $\Phi^{-1}(\R)$ and $\tilde{I}(\R)$ is its upper bound.
\end{prop}

\begin{proof}
For a non-zero ideal $I$ in $C^*(X)$ with $\R(I) = \R$, set $\mathring{I}=\overline{I\bigcap \mathbb{C}[X]}$. Since $\mathring{I}\subseteq I$, then $\mathcal{R}(\mathring{I})\subseteq \R$.
Conversely, for any $T\in I$ and $\varepsilon>0$, we have $T_\varepsilon\in \mathring{I}$ thanks to \cite[Theorem 3.2]{CW-rank}. This implies that $\R\subseteq \mathcal{R}(\mathring{I})$ and hence, $\R = \mathcal{R}(\mathring{I})$. Now applying Proposition \ref{R=R(I(R))}, we have $I(\R) = I(\R(\mathring{I})) = \mathring{I} \subseteq I$.

On the other hand, given $T \in I$ and $\varepsilon>0$, we have $\RANK(T_\varepsilon) \in \R(I) = \R$ by definition. Hence by Definition \ref{defn:ghostly ideal}, we have $T \in \tilde{I}(\R)$. Combining with Lemma \ref{lem:R for ghost}, we conclude the proof.
\end{proof}

%\begin{rem}
%Let $\mathcal{R}$ be a rank distribution on $X$, then for any ideal $I$ with $\mathcal{R}(I)=\mathcal{R}$, we have $I(\mathcal{R})\subseteq I\subseteq \widetilde{I}(\mathcal{R})$. This implies that $I(\mathcal{R})$ is minimal and $\widetilde{I}(\mathcal{R})$ is maximal among all of the ideals $I$ with $\mathcal{R}(I)=\mathcal{R}$.
%\end{rem}

\subsection{Invariant open subsets of $\beta X$ and structure of rank distributions}\label{ssec:inv open}
Now we focus on the lattice of rank distributions. 
To start, recall that a subset $U\subseteq \beta X$ is \emph{invariant} if any element $\gamma\in G(X)$ with source in $U$ also has its range in $U$. Given a rank distribution $\mathcal{R}$ on $X$, we define the \emph{associated subset} $U_\R \subseteq \beta X$ as follows:
\begin{equation}\label{EQ:inv open subset}
U_\R:=\bigcup_{\alpha\in \mathcal{R}}\overline{r(\supp(\alpha))},
\end{equation}
where $r: X \times X \to X$ is the range map. Similar to \cite[Lemma 5.2]{CW}, it is easy to see that $U_\R$ is a non-empty invariant open subset of $\beta X$.

%\begin{lem}
%Given a rank distribution $\R$ on $X$, the associated $U_\R$ is an invariant open subset of $\beta X$.
%\end{lem}

%While conversely, we can associate different rank distributions (and hence, different ideals) to a given invariant open subset $U\subseteq \beta X$. 

Following the outline in Section \ref{sec:intro}, we consider the following map in (\ref{EQ:Phi}):
%Similar to the map $\Psi$ defined in (\ref{EQ:fibring of ideals}), we consider:
\begin{equation*}
\Phi: \{\text{rank distributions on }X\} \longrightarrow \{\text{non-empty invariant open subsets of } \beta X\}
\end{equation*}
by $\R \mapsto U_{\R}$.
We will show that $\Phi$ is always surjective but \emph{not} injective in general, and explore the border of each $\Phi^{-1}(U)$.
%Hence to study the structure of rank distributions on $X$, it suffices to study the structure of each fibre $\Phi^{-1}(U)$. 

Let us start with the following:

\begin{defn}\label{defn:min and max}
Given a non-empty invariant open subset $U\subseteq \beta X$, we define the \emph{minimal (or bounded)} rank distribution associated to $U$ to be:
\[
\R_{\min}(U):=\{\alpha: \alpha \text{ is a bounded rank function such that } \overline{r(\supp(\alpha))}\subseteq U\},
\]
and the \emph{maximal} rank distribution associated to $U$ to be:
\[
\R_{\max}(U):=\{\alpha: \alpha \text{ is a rank function such that } \overline{r(\supp(\alpha))}\subseteq U\}.
\]
\end{defn}

The following shows that $\R_{\max}(U)$ and $\R_{\min}(U)$ are indeed rank distributions with the same associated invariant open subset and hence, $\Phi$ is surjective.

%It is clear that $\R_{\max}(U)$ is indeed a rank distribution on $X$ with $U_{\R_{\max}(U)} = U$. For $\R_{\min}(U)$, we also have:

\begin{lem}\label{bdd-rank}
Given a non-empty invariant open subset $U\subseteq \beta X$, then $\R_{\min}(U)$ and $\R_{\max}(U)$ are rank distributions on $X$ with $U_{\R_{\max}(U)} = U_{\R_{\min}(U)} = U$. 
\end{lem}

\begin{proof}
Given $\alpha\in \R_{\min}(U)$, it is clear that $\alpha^t\in \R_{\min}(U)$ since $U$ is invariant. Moreover, given an entourage $E\subseteq X\times X$, note that $r(\text{supp}(\alpha\cdot E))=r(\text{supp}(\alpha))$ and $E\cdot \alpha=(\alpha^t\cdot E^{-1})^t$. Hence we have $\alpha\cdot E\in \R_{\min}(U)$ and $E\cdot\alpha\in \R_{\min}(U)$. It is routine to check the other conditions in Definition \ref{def-rank distribution} and therefore, $\R_{\min}(U)$ is a rank distribution. By the same argument, $\R_{\max}(U)$ is also a rank distribution.

On the other hand, we have $U_{\R_{\min}(U)} \subseteq U_{\R_{\max}(U)} \subseteq U$ by construction. Conversely, given $u \in U$, take $\omega \in G(X)$ with $r(\omega) = u$. Note that $U$ is open, then we can find an entourage $E \subseteq X \times X$ such that $\omega \in \overline{E} \subseteq G(X)_{U}$ since all such $\overline{E}$ form a local basis of $\omega$. Taking $\alpha:=\chi_E$, we have $\alpha\in \R_{\min}(U)$ and $u \in \overline{r(\supp(\alpha))}$, which concludes the proof.
\end{proof}

To see that $\R_{\min}(U)$ and $\R_{\max}(U)$ provide the lower and upper bounds for the fibre $\Phi^{-1}(U)$, it is convenient to consult the tool of ideals in coarse structures from \cite{CW}. Recall that the \emph{bounded coarse structure} of $(X,d)$ is given by
\[
\E_d:=\{E \subseteq X \times X: E \text{ is an entourage}\}.
\]
Following \cite[Definition 4.1]{CW}, a non-empty \emph{ideal of $\E_d$} is a subset $\I \subseteq \E_d$ satisfying:
\begin{itemize}
 \item for any $A,B \subseteq \I$, then $A^{-1}$, $A \circ B$ and $A \cup B$ belong to $\I$;
 \item any finite subset of $X \times X$ belongs to $\I$;
 \item for any $A \in \I$ and $B \subseteq A$, then $B \in \I$;
 \item for any $A \in \I$ and $B \in \E_d$, then $A \circ B$ and $B \circ A$ belong to $\I$.
\end{itemize}
It was shown in \cite[Theorem 6.3]{CW} that there is a one-to-one correspondence between the set of all non-empty invariant open subsets of $\beta X$ and the set of all non-empty ideals of $\E_d$. More precisely, the invariant open set associated to a non-empty ideal $\I$ of $\E_d$ is given by
\[
U_{\I}:=\bigcup_{E \in \I} \overline{r(E)}.
\]
It is easy to see that under this correspondence, the map $\Phi$ defined in (\ref{EQ:Phi}) can be translated as follows:
\[
\{\text{rank distributions on }X\} \longrightarrow \{\text{non-empty ideals of } \E_d\}, \quad \R \mapsto \{\supp(\alpha): \alpha \in \R\}. 
\]
Now we prove the following sandwiching result for rank distributions:

%The following shows that $\R_{\min}(U)$ and $\R_{\min}(U)$ provide the lower and upper bounds for $\Phi^{-1}(U)$:

\begin{prop}\label{prop:sandwiching for rank dist}
For a non-empty invariant open subset $U\subseteq\beta X$ and a rank distribution $\R$ on $X$ with $U_\R = U$, we have $\R_{\min}(U) \subseteq \R \subseteq \R_{\max}(U)$. 
%Let $U\subseteq\beta X$ be an invariant open subset. Then for any rank distribution $\R$ on $X$ with $U_\R = U$, we have
%\[
%\R_{\min}(U) \subseteq \R \subseteq \R_{\max}(U).
%\]
Equivalently, $\R_{\min}(U)$ and $\R_{\max}(U)$ are the lower and upper bounds for the fibre $\Phi^{-1}(U)$.

Moreover, $\R_{\min}(U) = \R_{\max}(U)$ \emph{if and only if} $U=X$.
\end{prop}

\begin{proof}
It is clear from definition that $\R \subseteq \R_{\max}(U)$ for any $\R \in \Phi^{-1}(U)$. Now we focus on $\R_{\min}(U)$. 
By the analysis above together with Lemma \ref{bdd-rank}, we obtain
\[
\{\supp(\alpha): \alpha \in \R\} = \{\supp(\alpha): \alpha \in \R_{\min}(U)\} 
\]
for any $\R \in \Phi^{-1}(U)$. Therefore, given $\R \in \Phi^{-1}(U)$ and $\alpha \in \R_{\min}(U)$, we can always find $\alpha' \in \R$ with $\supp(\alpha') = \supp(\alpha)$. Since $\alpha$ is bounded, there exists $n\in \NN$ such that $\alpha \leq n \alpha'$. Finally, by condition (3) in Definition \ref{def-rank distribution}, we obtain that $\alpha \in \R$.

For the last statement, note that $\R_{\min}(X) = \R_{\max}(X)=\R_{\mathrm{fin}}$, the family of all finitely supported functions. Conversely, if $U \neq X$, then there exists an infinite subset $Y \subseteq X$ such that $\overline{Y}^{\beta X} \subseteq U$. Taking an unbounded $\ZZ_{\geq 0}$-valued function $\varphi$ on $Y$, we set $\alpha: X \times X \to \ZZ_{\geq 0}$ by $\alpha(x,y) = \varphi(x)$ if $x=y\in Y$, and $0$ otherwise. Then it is clear that $\alpha \in \R_{\max}(U)$ while $\alpha \notin \R_{\min}(U)$.
\end{proof}

\begin{ex}\label{ex:min}
For $U=\beta X$, we notice that $I(\R_{\min}(\beta X))$ is also called the uniform algebra of $X$, which plays an important role in \cite{Spa09}.
\end{ex}

\section{Relations between $I(\mathcal{R})$ and $\tilde{I}(\mathcal{R})$}\label{sec:relation}

As shown in Section \ref{ssec:ghostly ideals}, the ideal structure of the Roe algebra can be simplified once we know that the fibre $\Psi^{-1}(\R)$ consists of a single point for each rank distribution $\R$. Equivalently by Proposition \ref{sandwich lemma}, this suggests us to consider the question whether the geometric ideal $I(\R)$ coincides with the ghostly ideal $\tilde{I}(\R)$. 
In this section, we will focus on this question both for \emph{all} $\R$ in Section \ref{ssec:global version} and for a \emph{fixed} $\R$ in Section \ref{ssec:local version}, and prove Theorems \ref{thm:Ktheory intro}, \ref{thm:counterexample intro} and \ref{thm:partial A intro}.

\subsection{A global version}\label{ssec:global version}
It is known (at least mentioned in \cite{CW-rank}) that if the space has property A, then \emph{all} ideals in the Roe algebra are geometric. Here we provide an alternative and detailed proof using the tools developed above. 
%More precisely, we prove the following:

\begin{prop}\label{prop: property A case}
For a metric space $X$ of bounded geometry, the following are equivalent:
\begin{enumerate}
 \item $X$ has property A (or \emph{equivalently}, $G(X)$ is amenable);
 \item $I(\R)=\tilde{I}(\R)$ for any rank distribution $\R$ on $X$;
 \item all ideals in $C^*(X)$ are geometric;
 \item $I(X) = \tilde{I}(X)$.
\end{enumerate}
\end{prop}

To achieve, we introduce another picture for ghostly ideals. Given a rank distribution $\R$ on $X$, consider the following short exact sequence:
\[
0 \longrightarrow I(\R)\cap \ell^\infty(X;\K(\H)) \longrightarrow \ell^\infty(X;\K(\H)) \longrightarrow \frac{\ell^\infty(X;\K(\H))}{I(\mathcal{R})\cap \ell^\infty(X;\K(\H))} \longrightarrow 0,
\]
where $\ell^\infty(X;\K(\H))$ denotes the algebra of all bounded maps $X \to \K(\H)$. To simplify the notation, we denote 
\begin{equation}\label{EQ:A and J}
A:=\ell^\infty(X;\K(\H)) \quad \text{and} \quad J:=I(\R)\cap \ell^\infty(X;\K(\H)).
\end{equation}
We consider the natural action of $G(X)$ on $A$ from Example \ref{ex:coarse groupoid action}. Clearly, $J$ is invariant under the $G(X)$-action. Hence we have the following short sequence:
\begin{equation}\label{EQ:short seq}
0 \longrightarrow J \rtimes_r G(X) \longrightarrow A \rtimes_r G(X) \stackrel{\pi_\R}{\longrightarrow} (A/J)\rtimes_r G(X) \longrightarrow 0.
\end{equation}
Recall from \cite[Lemma 4.4]{STY} (see also Example \ref{ex:coarse groupoid action}) that we have $A \rtimes_r G(X) \cong C^*(X)$, and the same method can be applied to show that 
\[
J \rtimes_r G(X) \cong I(\R).
\] 
%It is easy to see that 
%\[
%J \rtimes_r G(X) \cong I(\R), \quad A \rtimes_r G(X) \cong C^*(X),
%\]
It is also easy to see that the map $\pi_\R$ is surjective. However in general, (\ref{EQ:short seq}) might \emph{not} be exact in the middle. In fact, we have:

\begin{prop}\label{kernel-ghost}
Given a rank distribution $\R$ on $X$, we have $\ker(\pi_{\R})=\tilde{I}(\mathcal{R})$. Hence we have the following short exact sequence:
\[
0 \longrightarrow \tilde{I}(\R) \longrightarrow C^*(X) \stackrel{\pi_\R}{\longrightarrow} (A/J)\rtimes_r G(X) \longrightarrow 0.
\]
\end{prop}

\begin{proof}
Recall from \cite[Lemma 9]{HLS} that the inclusion map $C_c(G(X); B)\rightarrow C_0(G(X); B)$ can be extended to an injection $i_B: B\rtimes_r G(X)\rightarrow C_0(G(X); B)$ for any $C^*$-algebra $B$ with a $G(X)$-action. Hence we have the following commutative diagram:
\[
	\xymatrixcolsep{3.2pc}\xymatrixrowsep{3.2pc}\xymatrix{
  A\rtimes_r G(X) \ar[d]_-{\textstyle i_A} \ar[r]^-{\textstyle \pi_\mathcal{R}} & (A/J)\rtimes_r G(X) \ar[d]^-{\textstyle i_{A/J}} \\
 C_0(G(X);A)\ar[r]^-{\textstyle q_\R}  & C_0(G(X);A/J),}
\]
where $q_\R$ is induced from the quotient map $A \to A/J$. It is easy to see that $\ker(q_\R) = C_0(G(X);J)$. 
Hence combining with Proposition \ref{R=R(I(R))}, we have
\begin{align*}
T\in \ker(\pi_\mathcal{R}) & \Longleftrightarrow i_A(T) \in C_0(G(X);J)\\
& \Longleftrightarrow T_\varepsilon\in C_c(G(X); J) \text{ for any }\varepsilon > 0\\
& \Longleftrightarrow T_\varepsilon\in I(\R) \text{ for any }\varepsilon > 0\\
& \Longleftrightarrow \RANK(T_\varepsilon) \in \mathcal{R}(I(\mathcal{R}))=\mathcal{R}\text{ for any }\varepsilon > 0\\
& \Longleftrightarrow T\in \tilde{I}(\mathcal{R})
\end{align*}
for any $T \in C^*(X)$, which concludes the proof.
\end{proof}

\begin{proof}[Proof of Proposition \ref{prop: property A case}]
``(1) $\Rightarrow$ (2)'': Let $\R$ be a rank distribution. Consider the following commutative diagram:
\begin{equation}\label{EQ:comm diag}
\xymatrix{
  0\ar[r] & J\rtimes G(X) \ar[d]_{\textstyle \pi} \ar[r] & A\rtimes G(X) \ar[d] \ar[r]& (A/J)\rtimes G(X) \ar[d] \ar[r] &0 \\
  0\ar[r] & \tilde{I}(\mathcal{R}) \ar[r] & A\rtimes_r G(X) \ar[r]& (A/J)\rtimes_r G(X) \ar[r] &0,
   }
\end{equation}
where $A$ and $J$ are defined in (\ref{EQ:A and J}). 
Here $\pi$ is the composition 
\[
\pi: J\rtimes G(X)\longrightarrow J\rtimes_r G(X) \cong I(\mathcal{R})\hookrightarrow \tilde{I}(\mathcal{R}). 
\]
The top horizon line is automatically exact, and the bottom line is also exact due to Proposition \ref{kernel-ghost}. 
Since $G(X)$ is amenable, the middle and right vertical lines are isomorphism by \cite[Corollary 5.6.17]{BO08}. Hence by the Five Lemma, $\pi$ is an isomorphism as well. Therefore, we conclude that condition (2).

``(2) $\Leftrightarrow$ (3)'' follows from Proposition \ref{sandwich lemma}, ``(3) $\Rightarrow$ (4)'' is trivial, and ``(4) $\Rightarrow$ (1)'' is a consequence of the main result of \cite{RW14}.
\end{proof}

Now we consider the $K$-theories of geometric and ghostly ideals. 

\begin{thm}\label{thm:K-theory}
If $X$ is a discrete metric space of bounded geometry which can be coarsely embedded into a Hilbert space, then for any rank distribution $\mathcal{R}$ on $X$, we have an isomorphism
\[
(\iota_\R)_*:K_*(I(\mathcal{R}))\longrightarrow K_*(\tilde{I}(\mathcal{R})) \quad \text{where} \quad \ast=0,1,
\] 
induced by the inclusion map $\iota_\R: I(\mathcal{R})\rightarrow \tilde{I}(\mathcal{R})$. Moreover, for any ideal $I$ in $C^*(X)$, we have an injective homomorphism 
\[
(\iota_I)_*:K_*(\overline{I\cap \mathbb{C}[X]})\longrightarrow K_*(I)  \quad \text{where} \quad \ast=0,1,
\]
induced by the inclusion map $\iota_I: \overline{I\cap \mathbb{C}[X]} \to I$.
\end{thm}

\begin{proof}
Given a rank distribution $\R$ on $X$, the diagram (\ref{EQ:comm diag}) induces a commutative diagram on the $K$-theory level:
\begin{scriptsize}
\begin{equation*}\label{EQ:K diag}
\xymatrix{
		\cdots \ar[r] & K_\ast(J\rtimes G(X)) \ar[r] \ar[d]_-{\textstyle \pi_\ast} & K_\ast(A\rtimes G(X)) \ar[r] \ar[d] & K_\ast((A/J)\rtimes G(X)) \ar[r] \ar[d] & K_{\ast+1}(J\rtimes G(X)) \ar[r] \ar[d]_{\textstyle \pi_{\ast+1}} &\cdots\\
		\cdots \ar[r] & K_\ast(\tilde{I}(\R)) \ar[r] & K_\ast(A\rtimes_r G(X)) \ar[r] & K_\ast((A/J)\rtimes_r G(X)) \ar[r] & K_{\ast+1}(\tilde{I}(\R)) \ar[r] &\cdots
	}
\end{equation*}
\end{scriptsize}
%\[
%\xymatrix{
%\vdots \ar[d] & \vdots \ar[d] \\
%K_*(J\rtimes G(X))\ar[d]\ar[r]^{\pi_*}& K_*(\widetilde{I}(\mathcal{R}))\ar[d]\\
%K_*(A\rtimes G(X)) \ar[d] \ar[r]& K_*(A\rtimes_r G(X)) \ar[d]\\
%K_*((A/J)\rtimes G(X)) \ar[d] \ar[r] & K_*((A/J)\rtimes_r G(X)) \ar[d]\\
%K_{*+1}(J\rtimes G(X)\ar[d]\ar[r] & K_{*+1}(\widetilde{I}(\mathcal{R}))\ar[d]\\
%\vdots & \vdots
%}
%\]
where both horizontal lines are exact.

Recall from \cite[Lemma 3.3]{STY} that there is a locally compact, second countable, ample and \'{e}tale groupoid $\mathcal{G}$ such that $G(X)=\beta X\rtimes \mathcal{G}$. It follows that $A\rtimes G(X)\cong A\rtimes \mathcal{G}$ and $A\rtimes_r G(X)\cong A\rtimes_r \mathcal{G}$. Note that the $\mathcal{G}$-$C^*$-algebra $A$ can be written as an inductive limit $\lim\limits_{\overrightarrow{i\in I}}A_i$ of separable $\mathcal{G}$-$C^*$-algebras. From \cite[Theorem 5.4]{STY}, we know that $\mathcal{G}$ is a-T-menable since $X$ is coarsely embeddable and hence, $\mathcal{G}$ is $K$-amenable thanks to \cite[Theorem 0.1]{Tu}. Then \cite[Proposition 4.12]{Tu} shows that the canonical map $K_*(A_i\rtimes \mathcal{G}) \to K_*(A_i\rtimes_r \mathcal{G})$ is an isomorphism for each $i$. As a result, we obtain $K_*(A\rtimes \mathcal{G}) \cong K_*(A\rtimes_r \mathcal{G})$ by taking limits and hence, $K_*(A\rtimes G(X)) \cong K_*(A\rtimes_r G(X))$. Similarly, we have $K_*(J\rtimes G(X)) \cong K_*(J\rtimes_r G(X))$ and $K_*((A/J)\rtimes G(X)) \cong K_*((A/J)\rtimes_r G(X))$. By the Five Lemma, we know $\pi_*$ is an isomorphism and hence $(\iota_{\R})_*$ is an isomorphism.

For the last statement, we assume that $I \neq 0$ without loss of generality. Setting $\mathcal{R}=\mathcal{R}(I)$, then $\overline{I\cap \mathbb{C}[X]}=I(\mathcal{R})$ by Proposition \ref{R=R(I(R))}. Now Proposition \ref{sandwich lemma} implies that $\iota_\R$ can be decomposed as $I(\mathcal{R})\subseteq I\subseteq \tilde{I}(\mathcal{R})$. Hence $(\iota_I)_*$ is injective since $(\iota_\R)_*$ is an isomorphism as shown above.
\end{proof}

\subsection{A local version}\label{ssec:local version}

Now we focus on a local version, \emph{i.e.}, the question when $I(\R) = \tilde{I}(\R)$ for a \emph{fixed} rank distribution $\R$ on $X$. To achieve this, we need the following notion of partial property A, introduced in \cite{WZ23}.

%we shall review the definition of partial property A. For a metric space $X$ and a rank distribution $\mathcal{R}$ on $X$, we will discuss the relations between $I(\mathcal{R})$ and $\widetilde{I}(\mathcal{R})$ when $X$ has partial property A.

\begin{defn}[{\cite[Definition 6.1]{WZ23}}]\label{defn:partial A}
Let $X$ be a discrete metric space of bounded geometry and $U \subseteq \beta X$ be an invariant open subset. We say that $X$ has \emph{partial property A towards $\beta X\setminus U$} if $G(X)_{\beta X\setminus U}$
is amenable.
\end{defn}

We record the following characterisation of property A for later use:

\begin{lem}[{\cite[Corollary 6.6]{WZ23}}]\label{equi-partial A}
Let $X$ be a discrete metric space of bounded geometry and $U\subseteq \beta X$ be an invariant open subset. Then the following are equivalent:
\begin{enumerate}
\item [(1)] $X$ has partial property A towards $U^c$.
\item [(2)] For any $\varepsilon, R>0$, there exists $S>0$, a subset $Y\subseteq X$ with $\overline{Y}^{\beta X}\subseteq U$ and a positive type kernel $\varphi(x,y):Y^c\times Y^c\rightarrow \mathbb{R}$ such that for any $x,y\in Y^c$,
\begin{itemize}
\item $\varphi(x,y)=\varphi(y,x)$ and $\varphi(x,x)=1$;
\item $|1-\varphi(x,y)|\leq \varepsilon$ if $d(x,y)\leq R$;
\item $\varphi(x,y)=0$ if $d(x,y)\geq S$.
\end{itemize}
\end{enumerate}
\end{lem}

Firstly, let us consider the case of maximal rank distributions.

\begin{prop}\label{prop: criteria to ensure I(R) = tilde I(R)}
Let $X$ be a discrete metric space of bounded geometry, $U$ be a non-empty invariant open subset of $\beta X$ and $\R:=\R_{\max}(U)$. If $X$ has partial property A towards $\beta X\setminus U$, then $I(\R) = \tilde{I}(\R)$.
\end{prop}

%\begin{prop}\label{prop: criteria to ensure I(R) = tilde I(R)}
%Let $X$ be a discrete metric space of bounded geometry and $U$ be an invariant open subset of $\beta X$. Denote $\R:=\R_{\max}(U)$. Then we have:
%\begin{enumerate}
% \item If $X$ has partial property A towards $\beta X\setminus U$, then $I(\R) = \tilde{I}(\R)$.
% \item If $G(X)_{\beta X \setminus U}$ is a-T-menable, then $(\iota_\R)_*:K_*(I(\mathcal{R}))\longrightarrow K_*(\tilde{I}(\mathcal{R}))$ is an isomorphism for $\ast=0,1$, where $\iota_\R: I(\mathcal{R})\rightarrow \tilde{I}(\mathcal{R})$ is the inclusion.
%\end{enumerate}
%\end{prop}

\begin{proof}
Consider the commutative diagram (\ref{EQ:comm diag}). Since $\R$ is maximal, we have:
\[
(A/J)\rtimes G(X) \cong (A/J)\rtimes G(X)_{\beta X \setminus U} \quad \text{and} \quad (A/J)\rtimes_r G(X) \cong (A/J)\rtimes_r G(X)_{\beta X \setminus U}.
\]
Note that the middle vertical line in (\ref{EQ:comm diag}) is always surjective and by assumption, the right vertical line is an isomorphism. Hence via a diagram chasing argument, we obtain that the left vertical line $\pi$ is surjective. Since $\pi$ factors through the inclusion map $\iota_{\R}: I(\R) \to \tilde{I}(\R)$, we conclude the proof.
\end{proof}

%Combining with Example \ref{ex:ghost} and Proposition \ref{prop:sandwiching for rank dist}, we recover the known fact that ghost operators are always compact provided $X$ has property A.

However, the situation becomes complicated for general rank distributions. In the following, we will show that partial property A \emph{cannot} ensure that $I(\R) = \tilde{I}(\R)$ even for minimal rank distributions. 

%While on the other hand, as we will show in the next subsection, partial property A \emph{cannot} ensure $I(\R) = \tilde{I}(\R)$ for an \emph{arbitrary} rank distribution. 

%\subsection{A counterexample}
%
%In \cite{WZ23}, the authors used partial property A to characterise when the geometric ideal coincides with the ghostly ideal in the case of uniform Roe algebra. While unfortunately as we will show in this subsection, this does \emph{not} hold in the case of Roe algebras. Roughly speaking, this is due to the fact that there might be different rank distributions associated to a given invariant open subset of $\beta X$ as shown in Proposition \ref{prop:sandwiching for rank dist}. 

%proved that for any invariant open subset $\mathcal{U}\subseteq \beta X$, the geometric ideal $I(\mathcal{U})$ and the ghost ideal $\widetilde{I}(\mathcal{U})$ in the uniform Roe algebra $C^*_u(X)$ are equal to each other if and only if $X$ has partial property A towards to $\mathcal{U}^c$. It is naturally to ask whether we have the similar result for the ideals in Roe algebras case. However, this is not necessary true since the ideals in Roe algebras depend not only on the invariant open subsets in $\beta X$ but also on the rank distributions. To see this, let's review the sequence of expander graphs at first.

\subsubsection{A counterexample}
Our example comes from expander graphs. Let us recall some basic notions.
Let $\mathcal{G}=(V,E)$ be an undirected finite graph with constant degree $\mathrm{degree}(\mathcal{G}) = D$, where  $V$ is the vertex set and $E$ is the edge set. For $v,w \in V$, denote $v \sim w$ if they are connected by an edge. Equip $V$ with the edge path metric.
The \emph{Laplacian} of $\mathcal{G}$ is a bounded linear operator 
$\Delta:\ell^2(V)\rightarrow \ell^2(V)$ defined by
$$\Delta(f)(v)=Df(v)-\sum\limits_{w\in V: w \sim v}f(w)=\sum\limits_{w\in V: w \sim v}(f(v)-f(w)) \quad \text{for} \quad v\in V.$$
%whence 
%$$\langle\Delta f,f\rangle=\sum\limits_{w \sim v}(f(v)-f(w))^2.$$
It is known that $\Delta$ is positive, $\ppg(\Delta)=1$ and the kernel of $\Delta$ consists of locally constant functions on $V$. Denote $\lambda_1(\Delta)$ the smallest positive eigenvalue of $\Delta$, and $P \in \B(\ell^2(V))$ the orthogonal projection to the kernel of $\Delta$.
%, which has the following matrix form:
%$$P=\dfrac{1}{\sharp V}
%\left(
%  \begin{array}{cccccccccccc}
%    1 & 1 & \cdots & 1 \\
%    1 & 1 & \cdots & 1  \\
%    \vdots & \vdots & \ddots & \vdots \\
%    1 & 1 & \cdots & 1 \\
%  \end{array}
%\right)_{V\times V}.
%$$

\begin{defn}
A sequence of finite graphs $\mathcal{G}_n=(V_n,E_n)$ of constant degree is called a sequence of \emph{expander graphs} if 
\begin{itemize}
\item $\exists K>0$ such that degree$(\mathcal{G}_n)\leq K$ for all $n$;
\item $\sharp V_n\rightarrow \infty$ as $n\rightarrow \infty$;
\item $\exists \varepsilon>0$ such that $\lambda_1(\Delta_n)\geq \varepsilon$ for all $n$.
\end{itemize}
\end{defn}

For a sequence of expander graphs $\mathcal{G}_n=(V_n,E_n)$, denote $d_n$ the edge path metric on $V_n$.
Let $X=\bigsqcup\limits_{n}(V_n,d_n)$ be their coarse disjoint union, \emph{i.e.}, equipping a metric $d$ on $X$ which is $d_n$ on each $V_i$ and $d(V_i, V_j)=i+j + \diam(V_i) + \diam(V_j)$ for $i\neq j$.
Here $\diam(V_i):=\sup\{d_i(v,w): v,w \in V_i\}$ is the diameter of $V_i$.
It is obvious that $(X,d)$ is a discrete metric space of bounded geometry. 

Define the Laplacian $\Delta:=\bigoplus\limits_{n}\Delta_n$ on $\ell^2(X)=\bigoplus\limits_{n}\ell^2(V_n)$, where $\Delta_n$ is the Laplacian on $\ell^2(V_n)$. Then $P=\bigoplus\limits_{n}P_n \in \B(\ell^2(X))$ is the orthogonal projection on the kernel of $\Delta$. Since $\Delta\in \CC[X]$ and $0$ is an isolated point in the spectrum $\sigma(\Delta)$, we have
\[
P=\chi_{\{0\}}(\Delta)\in C^*_u(X),
\]
where $\chi_{\{0\}}$ is the characteristic function of $\{0\}$. 

%The following result shows that partial property A is \emph{not} sufficient to ensure that the geometric ideal coincides with the ghostly one. 

The following result provides the required example:

\begin{thm}\label{prop:counterexample}
For a sequence of expander graphs $\mathcal{G}_n=(V_n,E_n)$, denote $(X,d)$ their coarse disjoint union. For $\R:=\R_{\min}(\beta X)$, we have:
%There exists a discrete metric space $(X,d)$ of bounded geometry which has partial property A towards $\beta X \setminus U$ for an invariant open subset $U \subseteq \beta X$ and a rank distribution $\R$ on $X$ with $U_{\R} = U$ satisfying:
\begin{enumerate}
 \item there exists a projection $Q \in \tilde{I}(\R)$ such that $\|Q-T\| \geq 1$ for any $T \in I(\R)$ and hence, $I(\R) \neq \tilde{I}(\R)$;
 \item the map $(\iota_{\R})_0: K_0(I(\R)) \to K_0(\tilde{I}(\R))$ is \emph{not} an isomorphism, where $\iota_{\R}: I(\R) \hookrightarrow \tilde{I}(\R)$ is the inclusion map.
\end{enumerate}
\end{thm}

\begin{proof}
%Taking a sequence of expander graphs $\mathcal{G}_n=(V_n,E_n)$, let $(X,d)$ be their coarse disjoint union. Let $U= \beta X$, and hence $X$ has partial property A towards $\beta X \setminus U$ trivially. Take $\R:=\R_{\min}(\beta X)$ to be the minimal rank distribution.
For each $n\in \NN$, denote $P_n$ the orthogonal projection onto constant functions in $\ell^2(V_n)$ and take $q_n$ to be an orthogonal projection on a $(\sharp V_n)^4$-dimensional subspace of the Hilbert space $\H$ in the definition of $C^*(X)$. Set $Q_n:=P_n \otimes q_n \in \B(\ell^2(V_n)) \otimes \B(\H) \subseteq \B(\ell^2(V_n)\otimes\H)$. Let $Q:=\bigoplus_n Q_n$, then $Q\in C^*(X)$. 

Given $\varepsilon>0$, $Q_\varepsilon$ is finitely supported and hence, we have $Q \in \tilde{I}(\mathcal{R})$. On the other hand, for any $T\in C_{\R}$, $\text{RANK}(T)$ is bounded by some $M$.  Since $T$ has finite propagation, we can choose $N\in \NN$ such that $T$ has the form $T=T_N\oplus(\bigoplus_{n>N}T_n)$, where $T_N\in B(\ell^2(\bigsqcup_{n\leq N}V_n)\otimes \H)$ and $T_n\in B(\ell^2(V_n)\otimes \H)$ for $n>N$. Choose $n_0>N$ such that $\sharp V_{n_0}>M$. Then by the choice of $q_{n_0}$, we can find some unit vector $v\in \H$ such that $q_{n_0}(v)=v$ and $T_{n_0}(x,y)(v)=0$ for all $(x,y)\in V_{n_0}\times V_{n_0}$. Let $\xi=\dfrac{1}{\sqrt{\sharp V_{n_0}}}(v,v,\cdots,v)^T\in \ell^2(V_{n_0})\otimes \H$, then $\|T_{n_0}-Q_{n_0}\|\geq \|(T_{n_0}-Q_{n_0})\xi\|=1$. Hence $\|T-Q\|\geq 1$, which implies that $Q\notin I(\mathcal{R})$.

Concerning $K$-theory, consider the following homomorphism from \cite{HLS}:
\[
\tau_0: K_0(C^*(X)) \longrightarrow \frac{\prod_n K_0(C^*(V_n))}{\bigoplus_n K_0(C^*(V_n))} \cong \frac{\prod \ZZ}{\bigoplus \ZZ}
\]
by taking the rank of each block. It is easy to see that $\tau_0([Q]) = [((\sharp V_n)^4)_{n\in \NN}]$. While on the other hand, we claim that for each $[p] \in K_0(I(\R))$, we can find a representative $(M_n)_{n\in \NN}$ for $\tau_0([p])$ such that the sequence $(M_n/(\sharp V_n)^3)_{n\in \NN}$ is bounded. This implies that $\tau_0([Q]) \notin \tau_0(K_0(I(\R)))$, which concludes (2). 

To see the claim, replacing $\H$ by its amplification, we assume that $p$ is a projection in $I(\R)$. Take $q\in C_{\R}$ with $\|p-q\|< \frac{1}{100}$. By definition, there exists $M>0$ such that $\rank(q(x,y)) \leq M$ for any $x,y\in X$. Hence for each $n\in \NN$, there exists a $2M \cdot (\sharp V_n)^2$-dimensional subspace $W_n \subseteq \H$ such that $\chi_{V_n} q \chi_{V_n} \in \B(\ell^2(V_n) \otimes W_n)$. Therefore for large $n$, $\chi_{V_n} q \chi_{V_n}$ is a quasi-projection and then by doing functional calculus, we obtain an actual projection $q_n \in \B(\ell^2(V_n) \otimes W_n) \subseteq C^*(V_n)$ close to $\chi_{V_n} q \chi_{V_n}$. Note that 
\[
\tau([p]) = [(\rank(q_n))_{n\in \NN}] \quad \text{and} \quad \rank(q_n) \leq 2M \cdot (\sharp V_n)^3.
\]
This concludes the claim and finishes the proof.
\end{proof}

%Note that in Theorem \ref{prop:counterexample}, $U_\R = \beta X$ and hence $X$ trivially has partial property A towards $\beta X \setminus U_{\R} = \emptyset$. 

\subsubsection{A criterion}
Theorem \ref{prop:counterexample} shows that partial property A is \emph{not} sufficient to ensure that the geometric and ghostly ideals are the same. In the following, we provide an extra condition to fill the gap.

Given a rank distribution $\mathcal{R}$, recall that $\mathcal{L}(\mathcal{R})$ is the associated spatial rank distribution defined in (\ref{EQ:L(R)}). Let
\[
\Gamma_\R:=\{\phi: X \to \FP(\H) \text{ such that } \rank \circ \phi \in \mathcal{L}(\mathcal{R})\},
\]
which forms a directed set by setting $\phi_1 \leq \phi_2$ if and only if the range of $\phi_1(x)$ is contained in the range of $\phi_2(x)$ for each $x\in X$. Here $\FP(\H)$ denotes the set of all finite rank projections on $\H$. 
For $\phi:X\rightarrow \FP(\H)$, denote the diagonal operator $\Lambda_\phi \in B(\ell^2(X) \otimes \H)$ by $\Lambda_\phi(x,x)=\phi(x)$ for $x\in X$. 
%Let
%\[
%\D_\R:=\{\Lambda_\phi: \phi \in \Gamma_\R\}.
%\] 

%Firstly, let us introduce some notation. Denote $\FP(\H)$ the set of all finite rank projections on $\H$. For any map $\phi:X\rightarrow \FP(\H)$, denote the diagonal operator $\Lambda_\phi \in B(\ell^2(X) \otimes \H)$ by $\Lambda_\phi(x,x)=\phi(x)$ for $x\in X$.
%
%Given a rank distribution $\mathcal{R}$, recall that $\mathcal{L}(\mathcal{R})$ is the associated spatial rank distribution defined in (\ref{EQ:L(R)}). Let
%\[
%\Gamma_\R:=\{\phi: X \to \FP(\H) \text{ such that } \rank \circ \phi \in \mathcal{L}(\mathcal{R})\},
%\]
%which forms a directed set by setting $\phi_1 \leq \phi_2$ if and only if the range of $\phi_1(x)$ is contained in the range of $\phi_2(x)$ for each $x\in X$. Also let
%\[
%\D_\R:=\{\Lambda_\phi: \phi \in \Gamma_\R\}.
%\] 
%which is a directed set in the sense that $\Lambda_{\phi_1}\leq \Lambda_{\phi_2}$ if and only if $\phi_1(x)\leq \phi_2(x)$ as projections on $\H$ for any $x\in X$.

\begin{lem}\label{lem:diagonal operator}
Given a rank distribution $\R$ on $X$ and $\phi\in \Gamma_\R$, we have $\Lambda_\phi \in I(\mathcal{R})$.
\end{lem}

\begin{proof}
By definition, $\Lambda_\phi$ is a diagonal operator and each entry has finite rank. Moreover, note that $\RANK(\Lambda_\phi) = \Delta_{\rank \circ \phi}$ (defined in Section \ref{ssec:spatial rank}). Since $\rank \circ \phi \in \L(\R)$, then by Lemma \ref{lem:spatial rank lem} we obtain $\RANK(\Lambda_\phi) \in \R$. This means that $\Lambda_\phi \in C_{\R}$ (see Definition \ref{defn:I(R)}) and hence, in $\Lambda_\phi \in I(\R)$.
\end{proof}

%\begin{lem}\label{lem:diagonal operator}
%Given a rank distribution $\R$ on $X$, $\phi\in \Gamma_\R$ and $T\in C^*(X)$, we have $\Lambda_\phi T\in I(\mathcal{R})$ and $T \Lambda_\phi\in I(\mathcal{R})$.
%\end{lem}
%
%\begin{proof}
%Since $\rank\circ \phi\in \mathcal{L(\mathcal{R})}$, it is obvious that $\RANK(\Lambda_\phi)\in \mathcal{R(\mathcal{L}(\mathcal{R)})} = \mathcal{R}$ thanks to Proposition \ref{relations between rank distributions}. Take a sequence of operators $\{T_n\}_n$ with finite propogation such that $T_n\rightarrow T$. For any $\varepsilon>0$, it is clear that $\Lambda_\phi \cdot (T_n)_\varepsilon \in \CC[X]$ and each of its matrix entry belongs to $\F(\H)$. Moreover, denote $E_n:=\supp(T_n)$ and a direct calculation shows that
%\[
%\RANK(\Lambda_\phi \cdot (T_n)_\varepsilon) \leq \RANK(\Lambda_\phi) \cdot E_n.
%\]
%Hence we obtain that $\RANK(\Lambda_\phi \cdot (T_n)_\varepsilon) \in \R$, which implies that $\Lambda_\phi \cdot (T_n)_\varepsilon \in C_\R$. Note that $\Lambda_\phi \cdot (T_n)_\varepsilon \to \Lambda_\phi \cdot T_n$ as $\varepsilon \to 0$ since $T_n$ has finite propagation, and hence $\Lambda_\phi \cdot T_n \in I(\R)$. Finally, letting $n\to \infty$, we obtain that $\Lambda_\phi T \in I(\R)$. The proof for $T \Lambda_\phi\in I(\R)$ is similar and hence, omitted.
%\end{proof}

%Now we are in the position to provide the following sufficient condition:

\begin{thm}\label{thm:partial A}
Let $X$ be a discrete metric space of bounded geometry, and $\R$ be a rank distribution on $X$ with $U:=U_\R$. Consider the following conditions:
\begin{enumerate}
\item [(1)] $X$ has partial property A towards to $\beta X \setminus U$;
\item [(2)] for any $T\in \tilde{I}(\mathcal{R}), \varepsilon>0$ and $Y \subseteq X$ with $\overline{Y}^{\beta X} \subseteq U$, there exists $\phi_0 \in \Gamma_\R$ such that 
\[
\|\Lambda_\phi \chi_Y T-\chi_{\supp(\phi)\cap Y} T\|<\varepsilon \quad \text{whenever} \quad \phi \geq \phi_0 \quad \text{in} \quad \Gamma_\R;
\]
%whenever $\phi \geq \phi_0$ in $\Gamma_\R$;
\item [(3)] $I(\mathcal{R})=\tilde{I}(\mathcal{R})$.
\end{enumerate}
Then $(1)+(2) \Rightarrow (3) \Rightarrow (2)$.
\end{thm}

\begin{proof} For ``$(1)+(2)\Rightarrow (3)$'': Given $T\in \tilde{I}(\mathcal{R})$ and $\varepsilon>0$, take $F\in \CC[X]$ with $\|T-F\|\leq \varepsilon$. Denote $R:=\ppg(F)$ and take $M>0$ such that $\sharp B(x,R)\leqslant M$ for each $x\in X$. By condition (1) and Lemma \ref{equi-partial A}, there exist $S>0$, $Y\subseteq X$ with $\overline{Y}^{\beta X}\subseteq U$ and a positive type kernel $\varphi(x,y): Y^c\times Y^c\rightarrow \mathbb{R}$ such that for any $x,y\in Y^c$, we have
\begin{enumerate}
 \item[(a)] $\varphi(x,y)=\varphi(y,x)$ and $\varphi(x,x)=1$; 
 \item[(b)] $|1-\varphi(x,y)|\leq \dfrac{\varepsilon}{M(\|T\|+1)}$ if $d(x,y)\leq R$;
 \item[(c)] $\varphi(x,y)=0$ whenever $d(x,y)\geq S$.
\end{enumerate}
Decomposing $T = T_1 + T_2 + T_3 + T_4$, where
\[
 T_1:=\chi_Y T \chi_Y, \quad T_2:=\chi_Y T \chi_{Y^c}, \quad T_3:=\chi_{Y^c} T \chi_Y \quad \text{ and } \quad T_4:=\chi_{Y^c} T \chi_{Y^c}.
\]
%and $F$ according to $X = Y \sqcup Y^c$ as follows:
%\[
%\begin{bmatrix}
%T_{1} & T_{2}\\
%T_{3} & T_{4}
%\end{bmatrix}
%\quad \text{and} \quad
%\begin{bmatrix}
%F_{1} & F_{2}\\
%F_{3} & F_{4}
%\end{bmatrix},
%\]
%where $T_1, F_1 \in \B(\ell^2(Y) \otimes \H)$ and $T_4, F_4 \in \B(\ell^2(Y^c) \otimes \H)$. 
Note that all $T_i$ belong to $\tilde{I}(\R)$ for $i=1,2,3,4$.

Consider the Schur multiplier $M_\varphi$ on $\B(\ell^2(Y^c) \otimes \H)$ defined by 
\[
(M_\varphi(A))(x,y):=\varphi(x,y)A(x,y) \quad \text{where} \quad x,y\in X,
\]
for $A\in \B(\ell^2(Y^c) \otimes \H)$. Then $\|M_\varphi\|\leq 1$ thanks to condition (a). Using the same argument for \cite[Lemma 4.3]{CW-05} and setting $F_4:= \chi_{Y^c} F \chi_{Y^c}$, we have 
\[
\|M_\varphi(T_4)-T_4\| \leq \|M_\varphi(T_4)-M_\varphi(F_4)\|+\|M_\varphi(F_4)-F_4\|+\|F_4-T_4\| \leq 3\varepsilon.
\]

On the other hand, since $\overline{Y}\subseteq U = U_\R$, there exists $\alpha\in \mathcal{R}$ with $r(\supp(\alpha))=Y$. Then $\chi_Y\in \mathcal{L(\mathcal{R})}$. Define $\phi_Y:X\rightarrow \FP(\H)$ by setting $\phi_Y(x)$ to be a rank-one projection on $\H$ if $x\in Y$, and $0$ otherwise. Then it is clear that $\rank \circ \phi_Y = \chi_Y$ and hence, $\phi_Y \in \Gamma_\R$. Also by condition (2), there exists $\phi_0\in \Gamma_\R$ such that 
\[
\|\Lambda_\phi \chi_Y T -\chi_{\supp(\phi)\cap Y} T\|<\varepsilon
\]
for any $\phi\geq \phi_0$ in $\Gamma_\R$ and hence,
\begin{equation}\label{EQ:app}
\|\Lambda_{\phi} T_1-\chi_{\supp(\phi)} T_1\|= \|(\Lambda_\phi \chi_Y T -\chi_{\supp(\phi)\cap Y} T)\chi_Y\|<\varepsilon. 
\end{equation}
Take $\phi_1\in \Gamma_\R$, which is greater than both $\phi_Y$ and $\phi_0$. Then $Y \subseteq \supp(\phi_1)$, which shows that $\chi_{\supp(\phi_1)} T_1 = T_1$. Combining with (\ref{EQ:app}), we obtain $\|\Lambda_{\phi_1} T_1-T_1\|<\varepsilon$. Similarly, there exist $\phi_2, \phi_3 \in \Gamma_\R$ such that $\|\Lambda_{\phi_2} T_2-T_2\|<\varepsilon$ and $\| T_3 \Lambda_{\phi_3} -T_3\|<\varepsilon$.
Combining them together, we obtain
\[
\|\Lambda_{\phi_1} T_1 + \Lambda_{\phi_2} T_2 + T_3\Lambda_{\phi_3} + M_{\varphi}(T_{4}) - T\| \leq 6\varepsilon.
\]
%\[
%\left\|\begin{bmatrix}
%\Lambda_{\phi_1} T_1 & \Lambda_{\phi_2} T_2\\
%T_3\Lambda_{\phi_3} & M_{\varphi}(T_{4})
%\end{bmatrix}-T\right\|\leq 5\varepsilon.
%\]
By Lemma \ref{lem:diagonal operator}, we have $\Lambda_{\phi_1} T_1$, $\Lambda_{\phi_2} T_2$ and $T_3\Lambda_{\phi_3}$ belong to $I(\R)$. By \cite[Theorem 3.3(iii)]{CW-rank}, $M_{\varphi}(T_{4}) \in \tilde{I}(\R) \cap \CC[X] \subseteq I(\R)$. Letting $\varepsilon \to 0$, we obtain that $T \in I(\R)$.

For ``$(3)\Rightarrow (2)$'': Given $T\in C_{\mathcal{R}}$ whose support is a partial translation, define a map $\phi_T: X \to \FP(\H)$ by setting $\phi_T(x)$ to be the orthogonal projection onto the range of $T(x,y)$ for each $x\in X$. Then $\phi_T \in \Gamma_\R$ since $T\in C_{\mathcal{R}}$. Also by definition, we have $\Lambda_\phi  T=T=\chi_{\supp(\phi)} T$ whenever $\phi \geq \phi_T$. 

For general $T \in C_\R$, decompose $T=T_1 + \cdots + T_n$ for some $n\in \NN$ such that each $T_i \in C_\R$ with $\supp(T_i)$ being a partial translation. Then for each $i=1,2,\cdots, N$, choose $\phi_{T_i} \in \Gamma_\R$ such that $\Lambda_{\phi}  T_i=T_i=\chi_{\supp(\phi)} T_i$ for any $\phi \geq \phi_{T_i}$. Taking $\phi_T\in\Gamma_\R$ greater than each $\phi_{T_i}$, we obtain $\Lambda_\phi  T=T=\chi_{\supp(\phi)} T$ whenever $\phi \geq \phi_T$. 

Finally for $T\in I(\mathcal{R})$, we choose a sequence $\{T_n\}\subseteq C_{\mathcal{R}}$ such that $\|T_n-T\|\to 0$ as $n\to \infty$. Hence for any $\varepsilon>0$, there exists $N\in \mathbb{N}$ such that $\|T_n-T\|\leq \varepsilon$ whenever $n\geq N$. Take $\phi_{T_N}$ as above such that $\Lambda_{\phi} T_N=T_N=\chi_{\supp(\phi)} T_N$ whenever $\phi \geq \phi_{T_N}$. Therefore for $\phi \geq \phi_{T_N}$, we have
\begin{align*}
\|\Lambda_{\phi} T - \chi_{\supp(\phi)} T\|&\leq \|\Lambda_{\phi}(T-T_N)\|+\|\chi_{\supp(\phi)} (T-T_N)\|+
\|\Lambda_{\phi_N} T_N-\chi_{\supp(\phi)} T_N\|\\
&\leq \|\Lambda_{\phi}\| \cdot \|(T-T_N)\|+\|\chi_{\supp(\phi)}\| \cdot \| (T-T_N)\| \leq 2\varepsilon.
\end{align*}
Hence for each $Y \subseteq X$ with $\overline{Y} \subseteq U$, we obtain $\|\Lambda_\phi \chi_Y T-\chi_{\supp(\phi)\cap Y} T\|<2\varepsilon$, which concludes (2) since $I(\mathcal{R})=\tilde{I}(\mathcal{R})$.
\end{proof}

In the proof of ``$(3)\Rightarrow (2)$'' above, we actually proved the following:
% stronger form of condition (2) holds for geometric ideals:

\begin{cor}
Let $\R$ be a rank distribution on $X$. Then for any $T\in I(\mathcal{R})$ and $\varepsilon>0$, there exists $\phi_0 \in \Gamma_\R$ such that  $\|\Lambda_\phi T-\chi_{\supp(\phi)} T\|<\varepsilon$ whenever $\phi \geq \phi_0$.
\end{cor}

Finally, we notice the following facts for the maximal rank distribution:

\begin{lem}\label{lem:condition for Rmax}
Let $U \subseteq \beta X$ be a non-empty invariant open subset, and $\R:=\R_{\max}(U)$ be the maximal rank distribution defined in Definition \ref{defn:min and max}. Given $T \in C^*(X)$, we have:
\begin{enumerate}
 \item $\chi_{\supp(\phi)} T \in I(\R)$ for any $\phi \in \Gamma_\R$;
 \item for any $\varepsilon>0$ and $Y \subseteq X$ with $\overline{Y}^{\beta X} \subseteq U$, there exists $\phi_0 \in \Gamma_\R$ such that whenever $\phi \geq \phi_0$ in $\Gamma_\R$, then $\|\Lambda_\phi \chi_Y T-\chi_{\supp(\phi)\cap Y} T\|<\varepsilon$.
% \[
% \|\Lambda_\phi \chi_Y T-\chi_{\supp(\phi)\cap Y} T\|<\varepsilon \quad \text{whenever} \quad \phi \geq \phi_0 \quad \text{in} \quad \Gamma_\R.
% \]
\end{enumerate}
\end{lem}

\begin{proof}
(1). For any $\epsilon>0$, choose $F \in \CC[X]$ such that $\|F-T\| < \epsilon$. Since $F$ has finite propagation, we can choose $\delta>0$ such that $\|F-F_{\delta}\| < \epsilon$. Hence
\[
\|\chi_{\supp(\phi)} T- \chi_{\supp(\phi)}F_{\delta}\| <2\epsilon.
\]
Note that $\chi_{\supp(\phi)}F_{\delta} \in C_\R \subseteq I(\R)$ since $\overline{\supp(\phi)} \subseteq U$ and $\R$ is the maximal rank distribution. 
Letting $\epsilon \to 0$, we conclude (1).

(2). We choose $F' \in \CC[X]$ such that $\|T-F'\| < \varepsilon/4$, and $\delta>0$ such that $\|F' - F'_\delta\| < \varepsilon/4$. Note that $F'_\delta$ has finite propagation and each entry has finite rank. Given $x \in Y$, we set $\phi_0(x) \in \FP(\H)$ whose range contains those of $F'_\delta(x,y)$ for all $y\in X$. Also set $\phi_0(x) = 0$ for $x\in X \setminus Y$. Hence $\supp(\phi_0) \subseteq Y$, which implies that $\phi_0 \in \Gamma_\R$ since $\R$ is maximal. Therefore for any $\phi \geq \phi_0$, we have
\[
\|\Lambda_\phi \chi_Y T-\chi_{\supp(\phi)\cap Y} T\| < \|\Lambda_\phi \chi_Y F'_\delta-\chi_{\supp(\phi)\cap Y} F'_\delta\| + \varepsilon = \varepsilon,
\]
which concludes the proof.
\end{proof}

%If $\mathcal{R}$ is a maximal rank distribution related to an invariant open set $\mathcal{U}\subseteq \beta X$, it is easy to check that condition (2) in the above theorem is automatically true. So we have the following corollary:

%Finally, we notice that in the case of maximal rank distributions, condition (2) in Theorem \ref{thm:partial A} is redundant:

Lemma \ref{lem:condition for Rmax}(2) shows that condition (2) in Theorem \ref{thm:partial A} always holds for maximal rank distribution. Hence combining with Theorem \ref{thm:partial A}, we obtain an alternative proof for Proposition \ref{prop: criteria to ensure I(R) = tilde I(R)}. 
However, according to Theorem \ref{prop:counterexample}, condition (2) does \emph{not} hold in general (at least not for bounded rank distributions).

%Consequently, we obtain 
%
%\begin{cor}
%Let $U \subseteq \beta X$ be an invariant open subset, and $\R:=\R_{\max}(U)$ be the maximal rank distribution defined in Definition \ref{defn:min and max}. If $X$ has partial property A towards to $\beta X \setminus U$, then $I(\mathcal{R})=\tilde{I}(\mathcal{R})$.
%\end{cor}

%\begin{proof}
%Using the same argument in the proof of Theorem \ref{thm:partial A} ``$(1)+(2)\Rightarrow (3)$'', we notice
%\[
%\left\|\begin{bmatrix}
%F_1 & F_2\\
%F_3 & M_{\varphi}(T_{4})
%\end{bmatrix}-T\right\|\leq 5\varepsilon.
%\]
%Since $F_1$ has finite propagation and $\supp(F_1) \subseteq Y \times Y$ for $\overline{Y}^{\beta X} \subseteq U$, we have $F_1 \in I(\R)$ due to the maximality of $\R$. Similarly, we have $F_2, F_3 \in I(\R)$. This concludes the proof.
%\end{proof}

%Taking $U=X$ and combining with Example \ref{ex:ghost} and Example \ref{ex:min}, we recover the following known result:
%
%\begin{cor}
%If $X$ has property A, then all ghost operators in $C^*(X)$ are compact.
%\end{cor}

\bigskip

%\begin{thebibliography}{10}

%\bibitem{AD}
%C. Anantharaman-Delaroche, \emph{Exact groupoids}, arXiv:1605.05117v2, 2021.
%
%\bibitem{CW}
%X. chen and Q. Wang, Ideal structure of uniform Roe algebras of coarse spaces, J. Funct. Anal. \textbf{216}(2004), 191-211.
%
%\bibitem{CW-05}
%X. Chen and Q. Wang, Ghost ideals in uniform Roe algebras of coarse spaces, Arch. Math. \textbf{84}(2005), 519-526. 
%
%\bibitem{CW-rank}
%X. Chen and Q. Wang, \emph{Rank distribution of coarse spaces and ideals structure of Roe algebras}, Bull. London Math. Soc. \textbf{38}(2006), 847-856.
%
%\bibitem{LeG}
%P. Y. Le Gall, \emph{Th\'{e}orie de Kasparov equivariante et groupo\"{\i}des}, Th\`{e}se de doctorat, Universit\'{e} de Paris VII, 1994.
%
%\bibitem{HLS}
%N. Higson, V. Lafforgue and G. Skandalis, \emph{Counterexamples to Baum-Connes conjectures}, Gemo. Funct. Anal. \textbf{12}(2002), 330-354.
%
%\bibitem{Re}
%J. Renault, \emph{A groupoid approach to $C^*$-algebras}, volume 793 of Lecture Notes in Mathematics, Springer, Berlin, 1980.
%
%\bibitem{STY}
%G. Skandalis, J. L. Tu and G. Yu, \emph{The coarse Baum-Connes conjecture and groupoids}, Topology. \textbf{41}(4)(2002), 807-834.
%
%\bibitem{Tu}
%J.-L. Tu, \emph{La conjecture de Baum-Connes pour les feuilletages moyennables}, K-theory, \textbf{17}(3)(1999), 215-264.
%
%\bibitem{WZ}
%Q. Wang and J. Zhang, Gostly ideas in uniform Roe algebras, arXiv:2301.04921, 2023.

\bibliographystyle{plain}
\bibliography{ghost}

\begin{thebibliography}{10}

\bibitem{AD}
C.~Anantharaman-Delaroche.
\newblock Exact groupoids.
\newblock {\em arXiv preprint arXiv:1605.05117}, 2016.

\bibitem{ADR00}
C.~Anantharaman-Delaroche and J.~Renault.
\newblock {\em Amenable groupoids}, volume~36 of {\em Monographies de
  L'Enseignement Math\'{e}matique}.
\newblock L'Enseignement Math\'{e}matique, Geneva, 2000.
\newblock With a foreword by Georges Skandalis and Appendix B by E. Germain.

\bibitem{BC00}
P.~Baum and A.~Connes.
\newblock Geometric {$K$}-theory for {L}ie groups and foliations.
\newblock {\em Enseign. Math.}, 46(1/2):3--42, 2000 (firstly circulated in
  1982).

\bibitem{BCH94}
P.~Baum, A.~Connes, and N.~Higson.
\newblock Classifying space for proper actions and {$K$}-theory of group
  {$C^*$}-algebras.
\newblock {\em Contemporary Mathematics}, 167:241--241, 1994.

\bibitem{BS87}
M.~Sh. Birman and M.~Z. Solomjak.
\newblock {\em Spectral theory of selfadjoint operators in {H}ilbert space}.
\newblock Mathematics and its Applications (Soviet Series). D. Reidel
  Publishing Co., Dordrecht, 1987.
\newblock Translated from the 1980 Russian original by S. Khrushch\"ev and V.
  Peller.

\bibitem{BO08}
Nathanial~P. Brown and Narutaka Ozawa.
\newblock {\em {$\rm C^*$}-algebras and finite-dimensional approximations},
  volume~88 of {\em Graduate Studies in Mathematics}.
\newblock American Mathematical Society, Providence, RI, 2008.

\bibitem{CW01}
X.~Chen and Q.~Wang.
\newblock Notes on ideals of {R}oe algebras.
\newblock {\em Q. J. Math.}, 52(4):437--446, 2001.

\bibitem{CW}
X.~Chen and Q.~Wang.
\newblock Ideal structure of uniform {R}oe algebras of coarse spaces.
\newblock {\em J. Funct. Anal.}, 216(1):191 -- 211, 2004.

\bibitem{CW-05}
X.~Chen and Q.~Wang.
\newblock Ghost ideals in uniform {R}oe algebras of coarse spaces.
\newblock {\em Arch. Math. (Basel)}, 84(6):519--526, 2005.

\bibitem{CW-rank}
X.~Chen and Q.~Wang.
\newblock Rank distributions of coarse spaces and ideal structure of {R}oe
  algebras.
\newblock {\em Bull. London Math. Soc.}, 38(5):847--856, 2006.

\bibitem{CWY13}
X.~Chen, Q.~Wang, and G.~Yu.
\newblock The maximal coarse {B}aum--{C}onnes conjecture for spaces which admit
  a fibred coarse embedding into {H}ilbert space.
\newblock {\em Adv. Math.}, 249:88--130, 2013.

\bibitem{HLS}
N.~Higson, V.~Lafforgue, and G.~Skandalis.
\newblock Counterexamples to the {B}aum-{C}onnes conjecture.
\newblock {\em Geom. Funct. Anal.}, 12(2):330--354, 2002.

\bibitem{HR95}
N.~Higson and J.~Roe.
\newblock On the coarse {B}aum-{C}onnes conjecture.
\newblock In {\em Novikov conjectures, index theorems and rigidity, {V}ol. 2
  ({O}berwolfach, 1993)}, volume 227 of {\em London Math. Soc. Lecture Note
  Ser.}, pages 227--254. Cambridge Univ. Press, Cambridge, 1995.

\bibitem{KS04}
M.~Khoshkam and G.~Skandalis.
\newblock Crossed products of {$C^*$}-algebras by groupoids and inverse
  semigroups.
\newblock {\em Journal of Operator Theory}, pages 255--279, 2004.

\bibitem{LeG}
P.-Y. Le~Gall.
\newblock {T}h{\'e}orie de {K}asparov {\'e}quivariante et groupo{\"\i}des.
\newblock {\em Comptes Rendus de l'Acad{\'e}mie des Sciences-Series
  I-Mathematics}, 324(6):695--698, 1997.

\bibitem{Re}
J.~Renault.
\newblock {\em A groupoid approach to {$C^{*}$}-algebras}, volume 793 of {\em
  Lecture Notes in Mathematics}.
\newblock Springer, Berlin, 1980.

\bibitem{Roe88}
J.~Roe.
\newblock An index theorem on open manifolds. {I}.
\newblock {\em J. Differential Geom.}, 27(1):87--113, 1988.

\bibitem{Roe96}
J.~Roe.
\newblock {\em Index theory, coarse geometry, and topology of manifolds},
  volume~90.
\newblock Amer. Math. Soc., 1996.

\bibitem{Roe03}
J.~Roe.
\newblock {\em Lectures on coarse geometry}, volume~31 of {\em University
  Lecture Series}.
\newblock Amer. Math. Soc., Providence, RI, 2003.

\bibitem{RW14}
J.~Roe and R.~Willett.
\newblock Ghostbusting and property {A}.
\newblock {\em J. Funct. Anal.}, 266(3):1674--1684, 2014.

\bibitem{STY}
G.~Skandalis, J.-L. Tu, and G.~Yu.
\newblock The coarse {B}aum-{C}onnes conjecture and groupoids.
\newblock {\em Topology}, 41(4):807--834, 2002.

\bibitem{Tu}
J.-L. Tu.
\newblock La conjecture de {B}aum-{C}onnes pour les feuilletages moyennables.
\newblock {\em $K$-Theory}, 17(3):215--264, 1999.

\bibitem{Spa09}
J.~\v{S}pakula.
\newblock Uniform {$K$}-homology theory.
\newblock {\em J. Funct. Anal.}, 257(1):88--121, 2009.

\bibitem{Wan07}
Q.~Wang.
\newblock Remarks on ghost projections and ideals in the {R}oe algebras of
  expander sequences.
\newblock {\em Arch. Math. (Basel)}, 89(5):459--465, 2007.

\bibitem{WZ23}
Q.~Wang and J.~Zhang.
\newblock Ghostly ideals in uniform {R}oe algebras.
\newblock {\em arXiv:2301.04921}, 2023.

\bibitem{WY12}
R.~Willett and G.~Yu.
\newblock Higher index theory for certain expanders and {G}romov monster
  groups, {I}.
\newblock {\em Adv. Math.}, 229(3):1380--1416, 2012.

\bibitem{Yu00}
G.~Yu.
\newblock The coarse {B}aum-{C}onnes conjecture for spaces which admit a
  uniform embedding into {H}ilbert space.
\newblock {\em Invent. Math.}, 139(1):201--240, 2000.

\end{thebibliography}

%\end{thebibliography}
%
%
%
%
%\bigskip
%
%\footnotesize

%\noindent  Benyin Fu \\
%College of Statistics and Mathematics,\\
%Shanghai Lixin University of Accounting and Finance,\\
%Shanghai 201209, P. R. China.\\
%E-mail: \url{fuby@lixin.edu.cn}\\

\end{document}